\documentclass[reqno, a4paper]{amsart}

\usepackage[english]{babel}
\usepackage{t1enc}
\usepackage{xcolor}
\usepackage{amscd,amsfonts}
\usepackage{amsmath,amsthm,amssymb,esint,enumerate}
\usepackage{txfonts,graphicx}
\usepackage{tikz}
\usetikzlibrary{arrows,shapes}

\theoremstyle{plain}
\newtheorem{thm}{Theorem}

\newtheorem{lem}[thm]{Lemma}

\theoremstyle{definition}
\newtheorem{definition}[thm]{Definition}

\newtheorem{ex}[thm]{Example}
\newcommand{\m}{\mathbf} 

\newcommand{\rd}{{/}}
\newcommand{\ld}{{\backslash}}
\newcommand{\ra}{\mathbin{\rightarrow}}
\newcommand{\jn}{\vee}
\newcommand{\mt}{\wedge}
\renewcommand{\ln}{{\sim}}

\newcommand{\da}{\mathord{\downarrow}}

\newcommand{\ls}{\setbox0\hbox{$-$}
\mathbin{\hbox{$-$\kern-\wd0\raise2\dp0\hbox{$\cdot$}\kern.3\wd0\lower2\dp0\hbox{$\cdot$}}}}
\newcommand{\rs}{\setbox0\hbox{$-$}
\mathbin{\hbox{$-$\kern-\wd0\lower2\dp0\hbox{$\cdot$}\kern.3\wd0\raise2\dp0\hbox{$\cdot$}}}}

\newcommand{\tw}{{\bf n}}
\newcommand{\drd}{\rotatebox[origin=c]{27.5}{$\Swarrow$}}
\newcommand{\dld}{\rotatebox[origin=c]{-27.5}{$\Searrow$}}

\allowdisplaybreaks

\begin{document}

\title{Twist structures and Nelson conuclei}
\author{Manuela Busaniche \and Nikolaos Galatos \and Miguel Marcos}

\begin{abstract}
Motivated by Kalman residuated lattices, Nelson residuated lattices and Nelson paraconsistent residuated lattices, we provide a natural common generalization of them. Nelson conucleus algebras unify these examples and further extend them to the non-commutative setting. We study their structure, establish a representation theorem for them in terms of twist structures and conuclei that results in a categorical adjunction, and explore situations where the representation is actually an isomorphism. In the latter case, the adjunction is elevated to a categorical equivalence. By applying this representation to the original motivating special cases we bring to the surface their underlying similarities.
\end{abstract}

\maketitle

\section*{Introduction}
Residuated lattices arise in many contexts in general and ordered algebra. Examples of residuated lattices include lattice-ordered groups, the lattice of ideals of a ring and relation algebras. At the same time they serve as algebraic semantics for substructural logics, including linear, relevance and many-valued logics. As a result, the algebraic semantics of these logics form further examples of residuated lattices and include MV, Heyting and Boolean algebras. In this paper we investigate a construction of involutive residuated lattices that has interesting applications to models of paraconsistent logics and use it to provide a unified approach to these models.

    Given a lattice ${\bf L}$, the \textbf{twist structure over} ${\bf L}$ is obtained by considering the direct product of  ${\bf L}$ and its order-dual ${\bf L}^\partial.$ The resulting lattice has a natural De Morgan involution given by $$ \sim (x, y) = (y,x)$$ for all $(x,y) \in L \times L^\partial$. This construction was used by Kalman in 1958 \cite{Kal}, while the modifier "twist" appeared thirty years later in  Kracht's paper \cite{Kra}. Although Kalman only worked with the lattice structure, several other authors considered expansions with additional operations on ${\bf L}$ which induce new and interesting operations on the twist structure  \cite{Fid,Vak,Cig86,Sen90,Kra,TsiWil,BusCig01,Odi04,Odi-book,BusCig00,BusCig02}.

    In particular, Tsinakis and Wille \cite{TsiWil}, inspired by Chu's work in category theory \cite{Bar} and its specialization to quantales \cite{Rosenthal}, considered the twist structure over a residuated lattice ${\bf  L}$ having a greatest element $\top$ and endowed it with a residuated lattice structure with unit $(e,\top)$, such that the pair $(\top, e)$ is the dualizing element for the natural involution.

    In \cite{BusCig00}, \cite{BusCig01} and \cite{BusCig02} it is proved that the logical systems of Nelson constructive logic with strong negation (CNS, see \cite{Nel}), and its paraconsistent  analogue (PNS, see \cite{Odi-book}) have as algebraic semantics residuated lattices whose lattice reducts are twist structures (see also \cite{SpiVer08a,SpiVer08b}). Furthermore, the monoid and residuum operators coincide with the ones proposed in \cite{TsiWil}. However, the unit of the residuated lattice is not the same in all the cases, so these structures do not fall directly under the framework of \cite{TsiWil}.

    Our aim is to present a unified approach that provides a deeper insight into the classes of residuated lattices that have a representation based on  twist structures. Our framework encompasses Nelson residuated lattices \cite{SpiVer08a,BusCig01}, Nelson paraconsistent residuated lattices \cite{BusCig00,BusCig02} and Kalman residuated lattices \cite{BusCig03,AgMar}. Our results allow  the comparison among them and provide some interesting new examples.

    To achieve this aim we start by considering a broader class of algebras: residuated lattice-ordered semigroups. Given a residuated lattice ${\bf L}$ we define the general twist-product ${\bf Tw(L)}$  as an involutive residuated lattice-ordered semigroup with an extra unary operation of involution; the fact that ${\bf L}$ may lack a top element results in ${\bf Tw(L)}$ potentially lacking an identity element. By localizing to a specific positive idempotent element of ${\bf Tw(L)}$ (induced by an arbitrary fixed element $\imath$ of ${\bf L}$), we obtain a subalgebra ${\bf Tw(L, \imath)}$ of ${\bf Tw(L)}$ that is a residuated lattice; this localization is done by the double-division conucleus given in \cite{GJ}, which focuses on the local submonoid of the positive idempotent element. This approach allows us to work with twist products of residuated lattices that do not have a top element and subsumes all of our motivating examples. Thus we accomplish our first goal:  we put into the same framework different algebras, such as Nelson residuated lattices and Paraconsistent Nelson residuated lattices.

    Having established this first theoretical framework, we pursue our second purpose: to describe the class of involutive residuated lattices that have a representation as a twist-product over a residuated lattice; this requires us to focus on a construction in the reverse direction to the twist product. We show that the desired residuated lattice is also obtained by a conucleus on the  involutive residuated lattice, which is of a very different nature than the double-division one, and we call it a \textbf{Nelson conucleus}. The main representation result in Theorem \ref{Teorema_representacion} shows that pairs of the form $({\bf A}, \tw )$, where ${\bf A}$ is a cyclic involutive residuated lattice and $\tw$ is a Nelson conucleus, are representable by a twist-product over a residuated lattice defined on $\tw[{\bf A}].$ We call these algebras \textbf{Nelson conucleus algebras} and denote the variety they form by $\mathcal{NCA}$.  As a corollary we provide an adjunction between the algebraic category given by $\mathcal{NCA}$ and a category whose objects are pairs of the form $({\bf L}, \imath)$ where ${\bf L}$ is a residuated lattice and $\imath\in L$ is a cyclic element. Along the way of proving the representation, we generalize the original construction of Rasiowa \cite{Rasi,Rasibook} on Nelson algebras and their representation by twist structures. In particular, our presentation shows that Rasiowa's homomorphic image construction can be replaced by a conucleus construction, which is more internal to the original algebra, as it provides representatives for the equivalence classes.

     Our motivating examples share some extra common features: they are commutative residuated lattices and the Nelson conucleus $\tw$ is definable by a term function; therefore they actually form varieties of commutative involutive residuated lattices. First we identify the subvariety of $\mathcal{NCA}$ whose elements are term equivalent to Kalman lattices. Secondly we    prove that Nelson residuated lattices and Paraconsistent Nelson residuated lattices form classes term equivalent to subvarieties of
    $\mathcal{NCA}$,  thus Theorem \ref{Teorema_representacion} applies to them. Furthermore, we show that both of these subvarieties are contained in $\mathcal{NT}$, a subvariety of $\mathcal{NCA}$, up to term equivalence, 
    whose elements we call \textbf{Nelson-type algebras}. Following Sendlewski's representation for Nelson algebras in \cite{Sen90} and the paraconsistent analogue given by Odintsov in \cite{Odi04}, we improve the representation of Theorem \ref{Teorema_representacion} for $\mathcal{NT}$  by providing a Sendlewski-like theorem, i.e.,  we identify each Nelson-type algebra with a subalgebra of the twist-product on $\tw[{\bf A}].$ As a consequence we get a categorical equivalence between the algebraic category of Nelson-type algebras and the category whose objects are triples of the form $({\bf H},i, F)$, where ${\bf H}$ is a Brouwerian algebra, $\imath$ is a cyclic element in ${\bf H}$ and $F$  is a Boolean filter of ${\bf H}$. This then restricts to categorical equivalences for the two subvarieties corresponding to Nelson residuated lattices and Paraconsistent Nelson residuated lattices.

    In the last section, we complete the circle of ideas by providing some particular conditions under which the representation Theorem \ref{Teorema_representacion} can be turned into an isomorphism theorem even when the Nelson conucleus is not given by a term.

\section{Preliminaries}

In this section we review some existing definitions and constructions. We also  introduce modifications and combinations of these constructions that will be suitable for the paper.

\subsection{Residuated lattices}

A \textbf{residuated lattice-ordered semigroup} is an algebra ${\bf A}=(A, \jn, \mt, \cdot, \ld, \rd)$ such that $(A,\cdot)$ is a semigroup, $(A,\vee, \wedge)$ is a lattice and the residuation condition
\begin{align*}x\cdot y\le z \mbox{ iff } y\le x\ld z\mbox{ iff } x\le z\rd y\end{align*}
holds for all $x, y$ and $z$ in $A$, where $\le$ is the order given by the lattice structure.
A \textbf{residuated lattice} is an algebra ${\bf A}=(A, \jn, \mt, \cdot, \ld, \rd, e)$ such that $(A,\cdot, e)$ is a monoid and $(A, \jn, \mt, \cdot, \ld, \rd)$ is a residuated lattice-ordered semigroup; residuated lattice-ordered semigroups and residuated lattices form varieties. If a residuated lattice-ordered semigroup satisfies $x\cdot y=y\cdot x$, it is called a \textbf{commutative}. In such case $x\ld y= y\rd x$ and we denote the common value by $x \to y$. We say that the residuated lattice ${\bf A}$ is \textbf{distributive} if the lattice  $(A,\vee, \wedge)$  is distributive. 
For each natural number $n$ we define $x^0=e$ and $x^n=x^{n-1}\cdot x$ for $n>0$; also we often write $xy$ for $x \cdot y$. For more on residuated lattices, see \cite{GJKO,BT,HarRafTsi}.

We will be also working with \textbf{integral residuated lattices}, which are residuated lattices satisfying $x \leq e$, and with \textbf{Brouwerian algebras} that are term equivalent to  residuated lattices satisfying $xy=x \wedge y$ (and hence are integral and commutative).
We also consider expansions with an additional constant, which serves as the bottom element, and refer to these algebras as \textbf{bounded residuated lattices}, since then they necessarily also have a term-definable top element.

\subsection{Involutive residuated lattices}
Residuated lattices can be expanded into (cyclic) involutive residuated lattices in two term equivalent ways. One is by adding to the signature a \textbf{negation constant} element $f$ that is \textbf{cyclic}: for all $x$, $x \ld f = f \rd x$, and \textbf{dualizing}: $f \rd (x \ld f) = x = (f \rd x) \ld f$. Alternatively, we can add to the signature a unary \textbf{(cyclic) involution} operation $\ln$ satisfying the equations:
 \begin{align*}\ln\ln x &= x  & \mbox{(double negation)}\\ x\ld \ln y &= \ln x\rd y & \mbox{(contraposition)}\end{align*}
The two ways produce term equivalent algebras via the definitions $f
= \ln e$ and $\ln x = x \ld f$ (see \cite{GalRaf,TsiWil} for
details); in particular $\ln x=x \ld \ln e$. We prefer the
definition in terms of the involution $\ln$, as it makes sense also
for residuated lattice-ordered semigroups. In this paper we will not
make use of the notion of not-necessarily cyclic involutive
residuated lattice, and we will always assume cyclicity for $f$ and
$\ln$. Observe that in commutative involutive residuated lattices
the \textbf{contraposition} takes the form $x \ra \ln y = y \ra  \ln
x$ and that cyclicity holds automatically.

In an involutive residuated lattice-ordered semigroup the divisions are definable in terms of the multiplication and the involution, and also the lattice operations are interdefinable via De Morgan equations. More precisely:

\begin{lem}\label{lem:propiedades_involucion} (see \cite[Lemma 5.1]{GalRaf} and \cite[Lemma 2.8]{GJKO}) If ${\bf A}$ is an involutive residuated lattice then
\begin{enumerate}
\item $x\ld y=\ln (\ln y\cdot x)$ and $y \rd x=\ln ( x\cdot \ln y)$.
\item $x\cdot y=\ln (y\ld \ln x)= \ln(\ln y\rd x)$.
\item $\ln (x\vee y) = \ln x \wedge \ln y$ and $\ln(x\wedge y) = \ln x\vee\ln y$.
\end{enumerate}
\end{lem}

An involutive residuated lattice is called \textbf{odd} if $f=e$, or equivalently $\ln e=e$; also, we have $\ln x  = x \ld e$. Because of the identification of the two constants, odd involutive residuated lattices are term equivalent to residuated lattices that satisfy the equations $x \ld e=e \rd x$ and $(x \ld e) \ld e=x$. In the commutative case, the defining equation is just   $(x\rightarrow e)\rightarrow e=x$.

\subsection{Conuclei constructions}

We review the important notion of conucleus on residuated structures.

\subsubsection{Conuclei and weak conuclei}

\begin{definition} A \textbf{weak conucleus} $\delta$ on a residuated lattice-ordered semigroup $\m A=(A, \jn, \mt, \cdot, \ld, \rd)$ is a function on $\m A$ that satisfies:
\begin{enumerate}[{(C}1{)}]
\item\label{Cord1} $\delta(x) \leq x$,
 \item\label{Cidem} $\delta(\delta(x))=\delta(x)$,
 \item\label{Cmono} if $x \leq y$ then $\delta(x) \leq \delta(y)$,
 \item\label{Cprod} $\delta(x)\cdot \delta(y) \leq \delta(x\cdot y)$.
 \end{enumerate}
 If ${\bf A}$ is a residuated lattice with neutral element $e$ and $\delta$ additionally satisfies:
 \begin{enumerate}[{(C}1{)}]
    \setcounter{enumi}{4}
 \item\label{Cneutral} $\delta(e)\cdot \delta(x)=\delta(x)\cdot \delta(e)=\delta(x),$
 \end{enumerate} then $\delta$ is called a \textbf{conucleus} on ${\bf A}.$
 \end{definition}

 Summing up, a conucleus is an interior operator $\delta$ on a residuated lattice ${\bf A}$  that satisfies (C\ref{Cprod}) and (C\ref{Cneutral}).
It is immediate that if $\delta$ is a weak conucleus then $\delta(x)\vee\delta(y)=\delta(\delta(x)\vee\delta(y))$ and $\delta(x)\cdot\delta(y)=\delta(\delta(x)\cdot\delta(y))$.

 Given a residuated lattice ${\bf A}$ and a conucleus $\delta$ on ${\bf A}$, the algebra  $$\m A_\delta=\delta[\m A]=(\delta[A], \jn, \mt_\delta, \cdot, \ld_\delta, \rd_\delta, \delta(e))$$ is a residuated lattice (\cite{GJKO}), where $x \mt_\delta y= \delta(x \mt y)$,  $x \ld_\delta y = \delta(x \ld y)$ and $y \rd_\delta x = \delta(y \rd x)$ for $x,y\in \delta[A]$ (by the previous observation, $\delta[A]$ is closed under $\vee$ and $\cdot$).

\begin{lem}\label{lem:propoconucleous} If $\delta$ is a weak conucleus on $\m A$, then
\begin{enumerate}[{(C}1{)}]
    \setcounter{enumi}{5}
    \item\label{Cinf} $\delta(x\wedge y) = \delta(\delta(x)\wedge \delta(y))$,
    \item\label{Cimpl} $\delta(\delta(x)\ld y) = \delta(\delta(x)\ld \delta(y))$ and
    $\delta(y \rd \delta(x)) = \delta(\delta(y)\rd \delta(x))$.
\end{enumerate}
\end{lem}

\begin{proof}
For (C\ref{Cinf}), observe that by (C\ref{Cord1}) we have that $\delta(x)\wedge\delta(y)\leq x\wedge y$, so  $\delta(\delta(x)\wedge\delta(y))\leq \delta(x\wedge y)$ by (C\ref{Cmono}). As $\delta(x\wedge y)\leq \delta(x),\delta(y)$ by (C\ref{Cmono}), we have that $\delta(x\wedge y)\leq \delta(\delta(x)\wedge\delta(y))$ again by (C\ref{Cmono}) and (C\ref{Cidem}).
For (C\ref{Cimpl}), as $\delta(x)\ld \delta(y)\leq \delta(x)\ld y$ we have one inequality. For the other one, using  (C\ref{Cprod}) we have  $\delta(x)\cdot \delta(\delta(x)\ld y)= \delta(\delta(x))\cdot \delta(\delta(x)\ld y)\leq \delta(\delta(x)\cdot \delta(x)\ld y)\leq  \delta(y)$ and (C\ref{Cimpl}) follows from  (C\ref{Cmono}) and  (C\ref{Cidem}).
\end{proof}

\subsubsection{Double Division conucleus}

Let $\m A=(A, \jn, \mt, \cdot, \ld, \rd)$ be a residuated lattice-ordered semigroup and $p\in A$ an idempotent element (i.e., $p = p^2$) that is also \textbf{positive} (i.e., $p \ld x, x \rd p \leq x \leq p x, x p$, for all $x$). This definition of being positive is justified by the fact that if $\m A$ has an identity element $e$, then $p$ is positive iff $e \leq p$. It is shown in \cite{GJ} that the map defined by $$\delta_p(x)=p \ld x \rd p$$ is a weak conucleus on $\m A$, that $\mt_p=\mt$, $\ld_p=\ld$, $\rd_p=\rd$, that $p$ is an identity element for $\delta_p[A]$ and that $\delta_p[A]=p \ld A \rd p =\{ p \ld a \rd p : a \in A\}$. Therefore the algebra $$p \ld \m A \rd p=\delta_p [{\bf A}]=(p \ld A \rd p, \mt, \jn, \cdot, \ld, \rd, p)$$ is a residuated lattice.  The algebra $p \ld \m A \rd p$ is called the \textbf{double-division conucleus image of $\m A$ by $p$}. Moreover, it turns out that $\delta_p[A]=\{a \in A : a p=a, p a=a\}$.

It is further shown in \cite{GJ} that if the residuated lattice-ordered semigroup $\m A$ is (cyclic) involutive with involution $\ln$, then the residuated lattice $p \ld \m A \rd p$ is also (cyclic) involutive and actually  $p \ld \m A \rd p$ is a subalgebra of $\m A$ with respect to the operations $\mt, \jn, \cdot, \ld, \rd, \ln$; the constants $e$ and $\ln e$ are replaced by $p$ and $\ln p$.

\section{Twist structures}

  Given a residuated lattice-ordered semigroup $\m L=(L, \mt, \jn, \cdot, \ld, \rd)$, we consider the set  $Tw(L)=L \times L$ and define the operations $\mt$ and $\jn$ as in $\m L \times \m L^\partial$. Also, for $a,a', b, b' \in L$, we define:
 \begin{equation}\label{eq:sim}
 \ln(a,b) = (b,a)
 \end{equation}
 \begin{equation}\label{eq:ast}
(a,b)\cdot (a',b') = (a \cdot a',b' \rd a \mt a'\ld b)
\end{equation}
\begin{equation}\label{eq:ld}
(a,b)\ld (a',b')    = (a\ld a'\mt b\rd b',b' \cdot a)
\end{equation}
\begin{equation}\label{eq:rd}
(a',b')\rd (a,b)    = (a'\rd a\mt b'\ld b,a \cdot b')
\end{equation}

The resulting structure
$(\text{Tw}(L), \mt, \jn, \cdot, \ld, \rd, \ln)$ is an involutive residuated lattice-ordered semigroup, which we denote by $\m{Tw}(\m L)$. We refer to this algebra as the \textbf{full twist structure} over $\m L$.
In addition, if $\m L$ is a topped residuated lattice with neutral element $e$ and top element $\top$ the structure $(\text{Tw}(L), \mt, \jn, \cdot, \ld, \rd, \ln, (e,\top))$
 is an involutive residuated lattice (see \cite{TsiWil}).
In particular, if $(L, \mt, \jn, \cdot, \ld, \rd, e)$ is an integral residuated lattice, the algebra $\m{Tw}(\m L,e)=(\text{Tw}(L), \mt, \jn, \cdot, \ld, \rd, \ln, (e,e))$ is an involutive residuated lattice.

\subsection{Double Division conucleus for twist structures}

Next we do not assume that $\m L$ has a top, so $\m{Tw}(\m L)$ may lack an identity element. 

\begin{lem} Let $\m L$ be a residuated lattice and $\imath \in L$.
\begin{enumerate}
    \item The element $(e,\imath)$ is a positive idempotent of the residuated lattice-ordered semigroup $\m{Tw}(\m L)$.
    \item An element $(a,b)$ is fixed by $\delta_{(e,\imath)}$ iff $ab \jn ba \leq \imath$. (Recall that $\delta_p(x)=p \ld x \rd p$.)
\end{enumerate}
\end{lem}

\begin{proof}
For (1) we have
$(e,\imath) (e,\imath)=(e\cdot e, \imath \rd e \mt e \ld \imath)=(e,\imath)$
so $(e,\imath)$ is a idempotent. To show that it is positive, we have that for all $(a,b)$,
\begin{align*}
    (a,b) (e,\imath) &= (ae, \imath \rd a \mt e \ld b) = (a, \imath \rd a \mt b) \geq (a,b)\\
    (e,\imath)(a,b)  &= (ea, b \rd e \mt a \ld \imath) = (a, b \mt a \ld \imath) \geq (a,b)\\
    (e,\imath) \ld (a, b) &= (e \ld a \mt \imath \rd b, be)= (a \mt \imath \rd b, b) \leq (a,b)\\
    (a, b) \rd (e,\imath) &= (a \rd e \mt b \ld \imath, be)= (a \mt b \ld \imath, b) \leq (a,b).
\end{align*}
Note that
\begin{align*}(e,\imath) \ld (a, b) \rd (e,\imath) &=
(e \ld a \mt \imath \rd b, be) \rd (e,\imath) =
(a \mt \imath \rd b, b) \rd (e,\imath) \\
&=((a \mt \imath \rd b)\rd e \mt b \ld \imath, eb) =
(a \mt \imath \rd b \mt b \ld \imath, b).\end{align*}
Therefore,  $(a,b)$ is fixed iff $(e,\imath) \ld (a, b) \rd (e,\imath) = (a,b)$ iff
$(a \mt \imath \rd b \mt b \ld \imath, b) = (a,b)$ iff
$a \leq \imath \rd b \mt b \ld \imath$ iff
$ab \leq \imath$ and $ba \leq \imath$ iff
$ab \jn ba \leq \imath$.
\end{proof}

 Given the fact that $(e,\imath)$ is a positive idempotent, for a residuated lattice $\m L$, we denote by
\begin{align*}\m{Tw}(\m L, \imath)=\delta_{(e,\imath)}[\m{Tw}(\m L)]=(e, \imath)\ld \m{Tw}(\m L) \rd (e, \imath)\end{align*}
 the double-division conucleus image of the residuated lattice-ordered semigroup $\m{Tw}(\m L)$ by $\delta(e,\imath)$. 
 Note that even though $\m{Tw}(\m L)$ may have no identity element, $\m{Tw}(\m L, \imath)$ is still a residuated lattice with identity element $(e,\imath)$. This allows us to consider this construction when $\m L$ is a residuated lattice that lacks a top element. We call $\m{Tw}(\m L, \imath)$ the \textbf{twist structure} over $(\m L, \imath)$

According to the previous results, the universe of the algebra $\m{Tw}(\m L, \imath)$ is the set $\text{Tw}(L,\imath) = \{(a,b)\in L\times L^\partial: ab \jn ba \leq \imath\}$. In summary we obtain the following:

\begin{thm}\label{twists} Let $(L, \mt, \jn, \cdot, \ld, \rd, e)$ be a residuated lattice and $\imath\in L$. Consider the set
$$\text{Tw}(L,\imath) = \{(a,b)\in L\times L^\partial: a b\vee b a\leq \imath\}$$
equipped with the operations $\wedge$ and $\vee$ of $\bf L\times L^\partial$ and the operations defined in equations (\ref{eq:sim}), (\ref{eq:ast}), (\ref{eq:ld}) and (\ref{eq:rd}).  Then ${\bf Tw(L,\imath)}=\left(\text{Tw}(L,\imath),\wedge,\vee,\cdot,\ld, \rd,\ln,(e,\imath)\right)$ is an involutive residuated lattice.
\end{thm}
We prove a lemma that sheds some light into the construction.

\begin{lem}
    The  set $$\text{Tw}(L,\imath) =
    \{(a,b)\in L\times L: ab \jn ba\leq \imath\}$$
    is a downset of the direct product lattice $\m L \times \m L$. Moreover, $M_\imath=\{(a,b): a = \imath \rd b \mt b \ld \imath  \text{ and } b =  \imath \rd a \mt a \ld \imath \}$ consists of maximal elements of $\text{Tw}(L,\imath)$ and if $\imath$ is cyclic then $\text{Tw}(L,\imath) = \da_{\m L \times \m L} M_\imath$. \end{lem}
\begin{proof} We prove that $ \text{Tw}(L,\imath)$ is a downset of ${\bf L}\times {\bf L}$ in terms of the coordinatewise order, which is different than the order of the residuated lattice ${\bf {Tw}(L,\imath)}.$
    If $(c,d), (a,b) \in L\times L $, are such that  $c \leq a$, $d \leq b$ and $ab \jn ba\leq \imath$, then $cd \jn dc \leq ab \jn ba\leq \imath$.
    Also, note that
\begin{align*}
        \text{Tw}(L,\imath) &=\{(a,b): a\le \imath \rd b \mt b \ld \imath  \text{ and } b \le  \imath \rd a \mt a \ld \imath\}\\
        &=\{(a,b): a\le \imath \rd b \mt b \ld \imath\}
        =\{(a,b): b \le  \imath \rd a \mt a \ld \imath\}.
\end{align*}
    It then follows directly that $\text{Tw}(L,\imath)$ is a downset. To show that the set $M_\imath=\{(a,b): a = \imath \rd b \mt b \ld \imath  \text{ and } b =  \imath \rd a \mt a \ld \imath \}$ consists of maximal elements of $\text{Tw}(L,\imath)$, let $(a, b)\in M_\imath$ and $(c,d) \in \text{Tw}(L,\imath)$, be such that  $a \leq c$, $b \leq d$. Then  $a = \imath \rd b \mt b \ld \imath$, $b =  \imath \rd a \mt a \ld \imath$, $c\le \imath \rd d \mt d \ld \imath$   and $d \le  \imath \rd c \mt c \ld \imath$. So, $c\le \imath \rd d \mt d \ld \imath \leq \imath \rd b \mt b \ld \imath = a \leq c$ and
    $d \le  \imath \rd c \mt c \ld \imath \leq \imath \rd a \mt a \ld \imath = b \leq d$, hence $(c,d)=(a,b)$.

    Assume now that $\imath$ is cyclic to prove that  $\text{Tw}(L,\imath) = \da_{\m L \times \m L} M_\imath$. Since $M_\imath \subseteq \text{Tw}(L,\imath)$ and $\text{Tw}(L,\imath)$ is a downset, we get $\da_{\m L \times \m L} M_\imath \subseteq \text{Tw}(L,\imath)$. Conversely, if $(c,d) \in \text{Tw}(L,\imath)$, then $c \leq (c\ld \imath) \ld \imath$, $d \leq  c\ld \imath $ and $((c\ld \imath) \ld \imath,c\ld \imath)\in M_\imath$; the element $(d \ld \imath, (d\ld \imath) \ld \imath)$ also witnesses this fact. This provides a visual understanding of how $\text{Tw}(L,\imath)$ sits inside $\text{Tw}(L)$.
\end{proof}

In Figure \ref{2_element_twists} we consider all possible twist structures and their maximal sets $M_\imath$ associated with the only two-element residuated lattice: the generalized Boolean algebra $\bf 2$. In Figure \ref{3_element_twists} we consider all possible twist structures and their maximal sets $M_\imath$ associated with all the three-element residuated lattices: the Wajsberg chain $\bf \text{\L}_3$, the G\"odel hoop $\bf G_3$ and the Sugihara monoid $\bf S_3$.

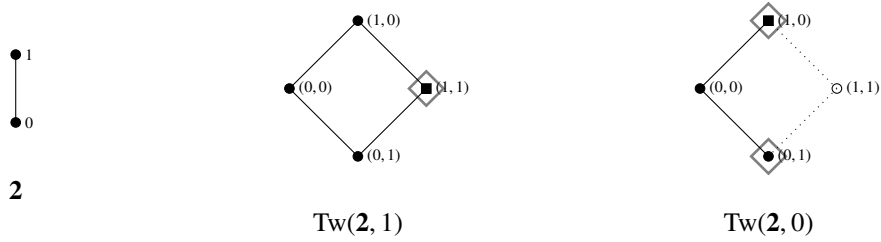
\begin{figure}[htp]
    \centering
    \begin{tikzpicture}[scale=0.9]
    \draw[black, fill=black] (0,0) circle (0.07 cm);
    \draw[black] (2,0) circle (0.07 cm);
    \draw[black, fill=black] (1,1) +(-2pt,-2pt) rectangle +(2pt,2pt) ;
    \node[diamond,gray,line width=1pt,draw]  at (1,1) {};
    \draw[black, fill=black] (1,-1) circle (0.07 cm);
    \node[diamond,gray,line width=1pt,draw]  at (1,-1) {};
    \draw[black] (0,0)--(1,1);
    \draw[black] (0,0)--(1,-1);
    \draw[dotted] (2,0)--(1,-1);
    \draw[dotted] (2,0)--(1,1);
    \node[right] at (0,0) {\tiny$(0,0)$};
    \node[right] at (2,0) {\tiny$(1,1)$};
    \node[right] at (1,1) {{\tiny$(1,0)$}};
    \node[right] at (1,-1) {\tiny$(0,1)$};
    \node at (0+1,-2) {$\text{Tw}({\bf 2},0)$};
    \draw[black, fill=black] (0-6,0) circle (0.07 cm);
    \draw[black, fill=black] (2-6,0) +(-2pt,-2pt) rectangle +(2pt,2pt) ;
    \node[diamond,gray,line width=1pt,draw]  at (2-6,0) {};
    \draw[black, fill=black] (1-6,1) circle (0.07 cm);
    \draw[black, fill=black] (1-6,-1) circle (0.07 cm);
    \draw[black] (0-6,0)--(1-6,1);
    \draw[black] (0-6,0)--(1-6,-1);
    \draw[black] (2-6,0)--(1-6,-1);
    \draw[black] (2-6,0)--(1-6,1);
    \node[right] at (0-6,0) {\tiny$(0,0)$};
    \node[right] at (2-6,0) {{\tiny$(1,1)$}};
    \node[right] at (1-6,1) {\tiny$(1,0)$};
    \node[right] at (1-6,-1) {\tiny$(0,1)$};
    \node at (0-6+1,-2) {$\text{Tw}({\bf 2},1)$};
    \draw[black, fill=black] (-10,-0.5+0) circle (0.07 cm);
    \draw[black, fill=black] (-10,-0.5+1) circle (0.07 cm);
    \draw[black] (-10,-0.5+0)--(-10,-0.5+1);
    \node[right] at (-10,-0.5+0) {\tiny$0$};
    \node[right] at (-10,-0.5+1) {\tiny$1$};
    \node at (-10,-0.5+-1) {$\bf 2$};
    \end{tikzpicture}
    \caption{The 2-element residuated lattice $\bf 2$, together with the lattice structure of all its possible twist structures, and the maximal sets that describe them. In each case, the maximal sets are shown in gray, and the identity element for the product $(1,\imath)$ is marked as a black square.}
    \label{2_element_twists}
\end{figure}

\begin{figure}[htp]
    \centering
    \begin{tikzpicture}[scale=0.7]
    \draw[black, fill=black] (0-6,-6+0) circle (0.07 cm);
    \node[diamond,gray,line width=1pt,draw]  at (0-6,-6+0) {};
    \draw[black, fill=black] (-2-6,-6+0) circle (0.07 cm);
    \draw[black] (2-6,-6+0) circle (0.07 cm);
    \draw[black, fill=black] (-1-6,-6+1) circle (0.07 cm);
    \draw[black] (1-6,-6+1) circle (0.07 cm);
    \draw[black, fill=black] (-1-6,-6+-1) circle (0.07 cm);
    \draw[black] (1-6,-6+-1) circle (0.07 cm);
    \draw[black, fill=black] (0-6,-6+2) +(-2pt,-2pt) rectangle +(2pt,2pt) ;
    \node[diamond,gray,line width=1pt,draw]  at (0-6,-6+2) {};
    \draw[black, fill=black] (0-6,-6+-2) circle (0.07 cm);
    \node[diamond,gray,line width=1pt,draw]  at (0-6,-6+-2) {};
    \draw[black] (0-6,-6+0)--(-1-6,-6+1);
    \draw[dotted] (0-6,-6+0)--(1-6,-6+1);
    \draw[black] (0-6,-6+0)--(-1-6,-6+-1);
    \draw[dotted] (0-6,-6+0)--(1-6,-6+-1);
    \draw[black] (-2-6,-6+0)--(-1-6,-6+-1);
    \draw[black] (-2-6,-6+0)--(-1-6,-6+1);
    \draw[dotted] (2-6,-6+0)--(1-6,-6+-1);
    \draw[dotted] (2-6,-6+0)--(1-6,-6+1);
    \draw[black] (0-6,-6+-2)--(-1-6,-6+-1);
    \draw[dotted] (0-6,-6+-2)--(1-6,-6+-1);
    \draw[black] (0-6,-6+2)--(-1-6,-6+1);
    \draw[dotted] (0-6,-6+2)--(1-6,-6+1);
    \node[right] at (0-6,-6+0) {\tiny$(a,a)$};
    \node[right] at (-2-6,-6+0) {\tiny$(0,0)$};
    \node[right] at (-1-6,-6+1) {\tiny$(a,0)$};
    \node[right] at (-1-6,-6+-1) {\tiny$(0,a)$};
    \node[right] at (0-6,-6+2) {{\tiny$(1,0)$}};
    \node[right] at (0-6,-6+-2) {\tiny$(0,1)$};
    \node at (0-6,-6+-3) {$\text{Tw}(\text{\L}_3,0)$};
    \draw[black, fill=black] (0-6,0) circle (0.07 cm);
    \draw[black, fill=black] (-2-6,0) circle (0.07 cm);
    \draw[black] (2-6,0) circle (0.07 cm);
    \draw[black, fill=black] (-1-6,1) circle (0.07 cm);
    \draw[black, fill=black] (1-6,1) +(-2pt,-2pt) rectangle +(2pt,2pt) ;
    \node[diamond,gray,line width=1pt,draw]  at (1-6,1) {};
    \draw[black, fill=black] (-1-6,-1) circle (0.07 cm);
    \draw[black, fill=black] (1-6,-1) circle (0.07 cm);
    \node[diamond,gray,line width=1pt,draw]  at (1-6,-1) {};
    \draw[black, fill=black] (0-6,2) circle (0.07 cm);
    \draw[black, fill=black] (0-6,-2) circle (0.07 cm);
    \draw[black] (0-6,0)--(-1-6,1);
    \draw[black] (0-6,0)--(1-6,1);
    \draw[black] (0-6,0)--(-1-6,-1);
    \draw[black] (0-6,0)--(1-6,-1);
    \draw[black] (-2-6,0)--(-1-6,-1);
    \draw[black] (-2-6,0)--(-1-6,1);
    \draw[dotted] (2-6,0)--(1-6,-1);
    \draw[dotted] (2-6,0)--(1-6,1);
    \draw[black] (0-6,-2)--(-1-6,-1);
    \draw[black] (0-6,-2)--(1-6,-1);
    \draw[black] (0-6,2)--(-1-6,1);
    \draw[black] (0-6,2)--(1-6,1);
    \node[right] at (0-6,0) {\tiny$(a,a)$};
    \node[right] at (-2-6,0) {\tiny$(0,0)$};
    \node[right] at (2-6,0) {\tiny$(1,1)$};
    \node[right] at (-1-6,1) {\tiny$(a,0)$};
    \node[right] at (1-6,1) {{\tiny$(1,a)$}};
    \node[right] at (-1-6,-1) {\tiny$(0,a)$};
    \node[right] at (1-6,-1) {\tiny$(a,1)$};
    \node[right] at (0-6,2) {\tiny$(1,0)$};
    \node[right] at (0-6,-2) {\tiny$(0,1)$};
    \node at (0-6,-3) {$\text{Tw}(\text{\L}_3,a)$};
    \draw[black, fill=black] (0-6,6+0) circle (0.07 cm);
    \draw[black, fill=black] (-2-6,6+0) circle (0.07 cm);
    \draw[black, fill=black] (2-6,6+0) +(-2pt,-2pt) rectangle +(2pt,2pt) ;
    \node[diamond,gray,line width=1pt,draw]  at (2-6,6+0) {};
    \draw[black, fill=black] (-1-6,6+1) circle (0.07 cm);
    \draw[black, fill=black] (1-6,6+1) circle (0.07 cm);
    \draw[black, fill=black] (-1-6,6+-1) circle (0.07 cm);
    \draw[black, fill=black] (1-6,6+-1) circle (0.07 cm);
    \draw[black, fill=black] (0-6,6+2) circle (0.07 cm);
    \draw[black, fill=black] (0-6,6+-2) circle (0.07 cm);
    \draw[black] (0-6,6+0)--(-1-6,6+1);
    \draw[black] (0-6,6+0)--(1-6,6+1);
    \draw[black] (0-6,6+0)--(-1-6,6+-1);
    \draw[black] (0-6,6+0)--(1-6,6+-1);
    \draw[black] (-2-6,6+0)--(-1-6,6+-1);
    \draw[black] (-2-6,6+0)--(-1-6,6+1);
    \draw[black] (2-6,6+0)--(1-6,6+-1);
    \draw[black] (2-6,6+0)--(1-6,6+1);
    \draw[black] (0-6,6+-2)--(-1-6,6+-1);
    \draw[black] (0-6,6+-2)--(1-6,6+-1);
    \draw[black] (0-6,6+2)--(-1-6,6+1);
    \draw[black] (0-6,6+2)--(1-6,6+1);
    \node[right] at (0-6,6+0) {\tiny$(a,a)$};
    \node[right] at (-2-6,6+0) {\tiny$(0,0)$};
    \node[right] at (2-6,6+0) {{\tiny$(1,1)$}};
    \node[right] at (-1-6,6+1) {\tiny$(a,0)$};
    \node[right] at (1-6,6+1) {\tiny$(1,a)$};
    \node[right] at (-1-6,6+-1) {\tiny$(0,a)$};
    \node[right] at (1-6,6+-1) {\tiny$(a,1)$};
    \node[right] at (0-6,6+2) {\tiny$(1,0)$};
    \node[right] at (0-6,6+-2) {\tiny$(0,1)$};
    \node at (0-6,6+-3) {$\text{Tw}(\text{\L}_3,1)$};
    \draw[black, fill=black] (-6,12+0) circle (0.07 cm);
    \draw[black, fill=black] (-6,12+1) circle (0.07 cm);
    \draw[black, fill=black] (-6,12-1) circle (0.07 cm);
    \draw[black] (-6,12+0)--(-6,12+1);
    \draw[black] (-6,12+0)--(-6,12-1);
    \node[right] at (-6,12+0) {\tiny$a$};
    \node[right] at (-6,12+1) {\tiny$1$};
    \node[right] at (-6,12-1) {\tiny$0=a^2$};
    \node at (-6,12-2) {$\text{\L}_3$};
    \draw[black] (0,-6+0) circle (0.07 cm);
    \draw[black, fill=black] (-2,-6+0) circle (0.07 cm);
    \draw[black] (2,-6+0) circle (0.07 cm);
    \draw[black, fill=black] (-1,-6+1) circle (0.07 cm);
    \draw[black] (1,-6+1) circle (0.07 cm);
    \draw[black, fill=black] (-1,-6+-1) circle (0.07 cm);
    \draw[black] (1,-6+-1) circle (0.07 cm);
    \draw[black, fill=black] (0,-6+2) +(-2pt,-2pt) rectangle +(2pt,2pt) ;
    \node[diamond,gray,line width=1pt,draw]  at (0,-6+2) {};
    \draw[black, fill=black] (0,-6+-2) circle (0.07 cm);
    \node[diamond,gray,line width=1pt,draw]  at (0,-6+-2) {};
    \draw[dotted] (0,-6+0)--(-1,-6+1);
    \draw[dotted] (0,-6+0)--(1,-6+1);
    \draw[dotted] (0,-6+0)--(-1,-6+-1);
    \draw[dotted] (0,-6+0)--(1,-6+-1);
    \draw[black] (-2,-6+0)--(-1,-6+-1);
    \draw[black] (-2,-6+0)--(-1,-6+1);
    \draw[dotted] (2,-6+0)--(1,-6+-1);
    \draw[dotted] (2,-6+0)--(1,-6+1);
    \draw[black] (0,-6+-2)--(-1,-6+-1);
    \draw[dotted] (0,-6+-2)--(1,-6+-1);
    \draw[black] (0,-6+2)--(-1,-6+1);
    \draw[dotted] (0,-6+2)--(1,-6+1);
    \node[right] at (-2,-6+0) {\tiny$(0,0)$};
    \node[right] at (-1,-6+1) {\tiny$(a,0)$};
    \node[right] at (-1,-6+-1) {\tiny$(0,a)$};
    \node[right] at (0,-6+2) {{\tiny$(1,0)$}};
    \node[right] at (0,-6+-2) {\tiny$(0,1)$};
    \node at (0,-6+-3) {$\text{Tw}(G_3,0)$};
    \draw[black, fill=black] (0,0) circle (0.07 cm);
    \draw[black, fill=black] (-2,0) circle (0.07 cm);
    \draw[black] (2,0) circle (0.07 cm);
    \draw[black, fill=black] (-1,1) circle (0.07 cm);
    \draw[black, fill=black] (1,1) +(-2pt,-2pt) rectangle +(2pt,2pt) ;
    \node[diamond,gray,line width=1pt,draw]  at (1,1) {};
    \draw[black, fill=black] (-1,-1) circle (0.07 cm);
    \draw[black, fill=black] (1,-1) circle (0.07 cm);
    \node[diamond,gray,line width=1pt,draw]  at (1,-1) {};
    \draw[black, fill=black] (0,2) circle (0.07 cm);
    \draw[black, fill=black] (0,-2) circle (0.07 cm);
    \draw[black] (0,0)--(-1,1);
    \draw[black] (0,0)--(1,1);
    \draw[black] (0,0)--(-1,-1);
    \draw[black] (0,0)--(1,-1);
    \draw[black] (-2,0)--(-1,-1);
    \draw[black] (-2,0)--(-1,1);
    \draw[dotted] (2,0)--(1,-1);
    \draw[dotted] (2,0)--(1,1);
    \draw[black] (0,-2)--(-1,-1);
    \draw[black] (0,-2)--(1,-1);
    \draw[black] (0,2)--(-1,1);
    \draw[black] (0,2)--(1,1);
    \node[right] at (0,0) {\tiny$(a,a)$};
    \node[right] at (-2,0) {\tiny$(0,0)$};
    \node[right] at (-1,1) {\tiny$(a,0)$};
    \node[right] at (1,1) {{\tiny$(1,a)$}};
    \node[right] at (-1,-1) {\tiny$(0,a)$};
    \node[right] at (1,-1) {\tiny$(a,1)$};
    \node[right] at (0,2) {\tiny$(1,0)$};
    \node[right] at (0,-2) {\tiny$(0,1)$};
    \node at (0,-3) {$\text{Tw}(G_3,a)$};
    \draw[black, fill=black] (0,6+0) circle (0.07 cm);
    \draw[black, fill=black] (-2,6+0) circle (0.07 cm);
    \draw[black, fill=black] (2,6+0) +(-2pt,-2pt) rectangle +(2pt,2pt) ;
    \node[diamond,gray,line width=1pt,draw]  at (2,6+0) {};
    \draw[black, fill=black] (-1,6+1) circle (0.07 cm);
    \draw[black, fill=black] (1,6+1) circle (0.07 cm);
    \draw[black, fill=black] (-1,6+-1) circle (0.07 cm);
    \draw[black, fill=black] (1,6+-1) circle (0.07 cm);
    \draw[black, fill=black] (0,6+2) circle (0.07 cm);
    \draw[black, fill=black] (0,6+-2) circle (0.07 cm);
    \draw[black] (0,6+0)--(-1,6+1);
    \draw[black] (0,6+0)--(1,6+1);
    \draw[black] (0,6+0)--(-1,6+-1);
    \draw[black] (0,6+0)--(1,6+-1);
    \draw[black] (-2,6+0)--(-1,6+-1);
    \draw[black] (-2,6+0)--(-1,6+1);
    \draw[black] (2,6+0)--(1,6+-1);
    \draw[black] (2,6+0)--(1,6+1);
    \draw[black] (0,6+-2)--(-1,6+-1);
    \draw[black] (0,6+-2)--(1,6+-1);
    \draw[black] (0,6+2)--(-1,6+1);
    \draw[black] (0,6+2)--(1,6+1);
    \node[right] at (0,6+0) {\tiny$(a,a)$};
    \node[right] at (-2,6+0) {\tiny$(0,0)$};
    \node[right] at (2,6+0) {{\tiny$(1,1)$}};
    \node[right] at (-1,6+1) {\tiny$(a,0)$};
    \node[right] at (1,6+1) {\tiny$(1,a)$};
    \node[right] at (-1,6+-1) {\tiny$(0,a)$};
    \node[right] at (1,6+-1) {\tiny$(a,1)$};
    \node[right] at (0,6+2) {\tiny$(1,0)$};
    \node[right] at (0,6+-2) {\tiny$(0,1)$};
    \node at (0,6+-3) {$\text{Tw}(G_3,1)$};
    \draw[black, fill=black] (0,12+0) circle (0.07 cm);
    \draw[black, fill=black] (0,12+1) circle (0.07 cm);
    \draw[black, fill=black] (0,12-1) circle (0.07 cm);
    \draw[black] (0,12+0)--(0,12+1);
    \draw[black] (0,12+0)--(0,12-1);
    \node[right] at (0,12+0) {\tiny$a=a^2$};
    \node[right] at (0,12+1) {\tiny$1$};
    \node[right] at (0,12-1) {\tiny$0$};
    \node at (0,12-2) {$G_3$};
    \draw[black] (0+6,-6+0) circle (0.07 cm);
    \draw[black, fill=black] (-2+6,-6+0) circle (0.07 cm);
    \draw[black] (2+6,-6+0) circle (0.07 cm);
    \draw[black, fill=black] (-1+6,-6+1) +(-2pt,-2pt) rectangle +(2pt,2pt) ;
    \draw[black] (1+6,-6+1) circle (0.07 cm);
    \draw[black, fill=black] (-1+6,-6+-1) circle (0.07 cm);
    \draw[black] (1+6,-6+-1) circle (0.07 cm);
    \draw[black, fill=black] (0+6,-6+2) circle (0.07 cm);
    \node[diamond,gray,line width=1pt,draw]  at (0+6,-6+2) {};
    \draw[black, fill=black] (0+6,-6+-2) circle (0.07 cm);
    \node[diamond,gray,line width=1pt,draw]  at (0+6,-6+-2) {};
    \draw[dotted] (0+6,-6+0)--(-1+6,-6+1);
    \draw[dotted] (0+6,-6+0)--(1+6,-6+1);
    \draw[dotted] (0+6,-6+0)--(-1+6,-6+-1);
    \draw[dotted] (0+6,-6+0)--(1+6,-6+-1);
    \draw[black] (-2+6,-6+0)--(-1+6,-6+-1);
    \draw[black] (-2+6,-6+0)--(-1+6,-6+1);
    \draw[dotted] (2+6,-6+0)--(1+6,-6+-1);
    \draw[dotted] (2+6,-6+0)--(1+6,-6+1);
    \draw[black] (0+6,-6+-2)--(-1+6,-6+-1);
    \draw[dotted] (0+6,-6+-2)--(1+6,-6+-1);
    \draw[black] (0+6,-6+2)--(-1+6,-6+1);
    \draw[dotted] (0+6,-6+2)--(1+6,-6+1);
    \node[right] at (-2+6,-6+0) {\tiny$(0,0)$};
    \node[right] at (-1+6,-6+1) {{\tiny$(e,0)$}};
    \node[right] at (-1+6,-6+-1) {\tiny$(0,e)$};
    \node[right] at (0+6,-6+2) {\tiny$(\top,0)$};
    \node[right] at (0+6,-6+-2) {\tiny$(0,\top)$};
    \node at (0+6,-6+-3) {$\text{Tw}(S_3,0)$};
    \draw[black, fill=black] (0+6,0) +(-2pt,-2pt) rectangle +(2pt,2pt) ;
    \node[diamond,gray,line width=1pt,draw]  at (0+6,0) {};
    \draw[black, fill=black] (-2+6,0) circle (0.07 cm);
    \draw[black] (2+6,0) circle (0.07 cm);
    \draw[black, fill=black] (-1+6,1) circle (0.07 cm);
    \draw[black] (1+6,1) circle (0.07 cm);
    \draw[black, fill=black] (-1+6,-1) circle (0.07 cm);
    \draw[black] (1+6,-1) circle (0.07 cm);
    \draw[black, fill=black] (0+6,2) circle (0.07 cm);
    \node[diamond,gray,line width=1pt,draw]  at (0+6,2) {};
    \draw[black, fill=black] (0+6,-2) circle (0.07 cm);
    \node[diamond,gray,line width=1pt,draw]  at (0+6,-2) {};
    \draw[black] (0+6,0)--(-1+6,1);
    \draw[dotted] (0+6,0)--(1+6,1);
    \draw[black] (0+6,0)--(-1+6,-1);
    \draw[dotted] (0+6,0)--(1+6,-1);
    \draw[black] (-2+6,0)--(-1+6,-1);
    \draw[black] (-2+6,0)--(-1+6,1);
    \draw[dotted] (2+6,0)--(1+6,-1);
    \draw[dotted] (2+6,0)--(1+6,1);
    \draw[black] (0+6,-2)--(-1+6,-1);
    \draw[dotted] (0+6,-2)--(1+6,-1);
    \draw[black] (0+6,2)--(-1+6,1);
    \draw[dotted] (0+6,2)--(1+6,1);
    \node[right] at (0+6,0) {{\tiny$(e,e)$}};
    \node[right] at (-2+6,0) {\tiny$(0,0)$};
    \node[right] at (-1+6,1) {\tiny$(e,0)$};
    \node[right] at (-1+6,-1) {\tiny$(0,e)$};
    \node[right] at (0+6,2) {\tiny$(\top,0)$};
    \node[right] at (0+6,-2) {\tiny$(0,\top)$};
    \node at (0+6,-3) {$\text{Tw}(S_3,e)$};
    \draw[black, fill=black] (0+6,6+0) circle (0.07 cm);
    \draw[black, fill=black] (-2+6,6+0) circle (0.07 cm);
    \draw[black, fill=black] (2+6,6+0) circle (0.07 cm);
    \node[diamond,gray,line width=1pt,draw]  at (2+6,6+0) {};
    \draw[black, fill=black] (-1+6,6+1) circle (0.07 cm);
    \draw[black, fill=black] (1+6,6+1) circle (0.07 cm);
    \draw[black, fill=black] (-1+6,6+-1) circle (0.07 cm);
    \draw[black, fill=black] (1+6,6+-1) +(-2pt,-2pt) rectangle +(2pt,2pt) ;
    \draw[black, fill=black] (0+6,6+2) circle (0.07 cm);
    \draw[black, fill=black] (0+6,6+-2) circle (0.07 cm);
    \draw[black] (0+6,6+0)--(-1+6,6+1);
    \draw[black] (0+6,6+0)--(1+6,6+1);
    \draw[black] (0+6,6+0)--(-1+6,6+-1);
    \draw[black] (0+6,6+0)--(1+6,6+-1);
    \draw[black] (-2+6,6+0)--(-1+6,6+-1);
    \draw[black] (-2+6,6+0)--(-1+6,6+1);
    \draw[black] (2+6,6+0)--(1+6,6+-1);
    \draw[black] (2+6,6+0)--(1+6,6+1);
    \draw[black] (0+6,6+-2)--(-1+6,6+-1);
    \draw[black] (0+6,6+-2)--(1+6,6+-1);
    \draw[black] (0+6,6+2)--(-1+6,6+1);
    \draw[black] (0+6,6+2)--(1+6,6+1);
    \node[right] at (0+6,6+0) {\tiny$(e,e)$};
    \node[right] at (-2+6,6+0) {\tiny$(0,0)$};
    \node[right] at (2+6,6+0) {\tiny$(\top,\top)$};
    \node[right] at (-1+6,6+1) {\tiny$(e,0)$};
    \node[right] at (1+6,6+1) {\tiny$(\top,e)$};
    \node[right] at (-1+6,6+-1) {\tiny$(0,e)$};
    \node[right] at (1+6,6+-1) {{\tiny$(e,\top)$}};
    \node[right] at (0+6,6+2) {\tiny$(\top,0)$};
    \node[right] at (0+6,6+-2) {\tiny$(0,\top)$};
    \node at (0+6,6+-3) {$\text{Tw}(S_3,\top)$};
    \draw[black, fill=black] (+6,12+0) circle (0.07 cm);
    \draw[black, fill=black] (+6,12+1) circle (0.07 cm);
    \draw[black, fill=black] (+6,12-1) circle (0.07 cm);
    \draw[black] (+6,12+0)--(+6,12+1);
    \draw[black] (+6,12+0)--(+6,12-1);
    \node[right] at (+6,12+0) {\tiny$e$};
    \node[right] at (+6,12+1) {\tiny$\top$};
    \node[right] at (+6,12-1) {\tiny$0$};
    \node at (+6,12-2) {$S_3$};
    \end{tikzpicture}
    \caption{The 3-element residuated lattices $\bf{\text{\L}_3}$, $\bf G_3$ and $\bf S_3$, together with the lattice structure of all their possible twist structures, and the maximal sets that describe them. In each case, the maximal sets are shown in gray, and the identity element for the product $(1,\imath)$ is marked as a square.}
    \label{3_element_twists}
\end{figure}

\subsection{Motivating examples}

We consider three different varieties of algebras that serve as the motivation for our study. The goal is to include them under the same theoretical framework.

\begin{ex}{}         A  \textbf{Kalman residuated lattice}  (\cite{BusCig03,AgMar}) is a commutative residuated lattice  ${\bf A}=(A, \wedge, \vee, \cdot, \rightarrow, e)$ satisfying the following equations:
        \begin{enumerate}[{(K}1{)}]
        \item \label{Kinvo}$(x\rightarrow e)\rightarrow e=x$,
        \item \label{Kdist_ast_inf}$(x\cdot y)\wedge e= (x\wedge e)\cdot (y\wedge e)$,
        \item \label{Kquasi_nelson} $((x\wedge e)\rightarrow y)\wedge (x \ra (y \jn e))=x\rightarrow y$ [equivalently, $((x\wedge e)\rightarrow y)\wedge ((\ln y\wedge e)\rightarrow
        \ln x)=x\rightarrow y$],
        \item \label{Kquasidist} $e \wedge (x \vee y) = (e \wedge x) \vee (e \wedge y)$ and
         \item \label{KquasidistUnnecessary}  $x \land (y \lor e) = (x \land y) \lor (x \land e).$
    \end{enumerate}

$(A, \wedge, \vee, \cdot, \rightarrow, \ln, e)$
is a commutative involutive residuated lattice, where $\ln x=x\rightarrow e$ and $\sim e=e$, i.e., $\m A$ is an odd residuated lattice.  Note that the expressions $(\ln y\wedge e)\rightarrow \ln x$ and $x \ra (y \jn e)$ in the two forms of (K\ref{Kquasi_nelson}) are equal by contraposition. It will follow from our analysis that  (K\ref{KquasidistUnnecessary}) is actually redundant.

 If ${\bf L}=(L, \vee, \wedge,\cdot, \to, 1)$ is an integral commutative residuated lattice, then ${\bf Tw}({\bf L},1)$ is a Kalman lattice. Moreover, for each Kalman lattice ${\bf A}$ there is an integral residuated lattice ${\bf L}$ such that ${\bf A}$ is isomorphic to a subalgebra of ${\bf Tw}({\bf L},1)$ \cite[Th. 2.5]{BusCig03}.
\end{ex}

\begin{ex}   A \textbf{Nelson residuated lattice}  (\cite{SpiVer08a,SpiVer08b,BusCig01,Fid,Sen90}) is a bounded integral  commutative residuated lattice ${\bf A}=(A, \vee, \wedge, \cdot, \to, \bot, e)$, that with $\neg x:=x\to \bot$ satisfies:
        \begin{enumerate}[{(NRL}1{)}]
            \item $\neg \neg x=x$.
            \item $(x^2\to y)\wedge ((\neg y)^2\to \neg x)=x\to y$.\end{enumerate}
 Involutive residuated lattices defined in \cite{BusCig01} are not the same as the ones defined here, since the involution $\neg$ is a definable operation if the constant $\bot$ is included in the type.
Nelson residuated lattices are term equivalent to Nelson algebras \cite{Rasi,Rasibook}, the algebraic counterpart of Nelson constructive logic with strong negation \cite{Nel}.

If ${\bf H}=(H, \vee, \wedge, \to, 0,1)$ is a Heyting algebra then viewing $\m H$ as a bounded residuated lattice with  $\cdot=\wedge$ and bottom element $0$, ${\bf Tw}({\bf H},0)$ is a Nelson residuated lattice with bottom element  $(0,1)$, which is added to the signature of ${\bf Tw}({\bf H},0)$. As in the previous case, there is a representation theorem since every Nelson residuated lattice is embeddable in  ${\bf Tw}({\bf H},0)$ for a Heyting algebra ${\bf H}$ (see \cite[Corollary 3.5]{BusCig01}).
\end{ex}

\begin{ex}
 A \textbf{Nelson paraconsistent residuated lattice} (NPc-lattice for short, see \cite{BusCig00}, \cite{BusCig02}, \cite{Odi-book} and \cite{ABGM}) is an odd distributive commutative residuated lattice  ${\bf A}=(A, \vee, \wedge, \cdot, \to, e)$ satisfying, for $\ln x = x\ra e$:
        \begin{enumerate}[{(NPc}1{)}]
            \item $\sim\sim x=x$
            \item $(x\cdot y)\wedge e = (x\wedge e)\cdot(y\wedge e)$
            \item $(x\wedge e)^2 = x\wedge e$
            \item $((x\wedge e)\rightarrow y)\wedge (x \ra (y \jn e))=x\rightarrow y$ [equivalently, $((x\wedge e)\rightarrow y)\wedge ((\ln y\wedge e)\rightarrow
        \ln x)=x\rightarrow y$]
        \end{enumerate}

NPc-lattices are special Kalman residuated lattices satisfying distributivity and (NPc3).
They were introduced in \cite{BusCig00} in order to present the algebraic semantics of Nelson paraconsistent logic \cite{Odi-book} within the framework of residuated lattices.

   If ${\bf H}=(H, \vee, \wedge, \to, 1)$ is a Brouwerian algebra (also called generalized Heyting algebra), then viewing $\m H$ as a residuated lattice with  $\cdot=\wedge$,  ${\bf Tw}({\bf H},1)$ is a Nelson paraconsistent residuated lattice. As in the previous cases, every NPc-lattice ${\bf A}$ can be embedded into a twist structure  ${\bf Tw}({\bf H},1)$ (see \cite{ABGM}).
\end{ex}

\subsection{Natural conuclei on twist structures}\label{section_natural_conuclei}

Another conucleus  will be important for understanding
involutive residuated lattices represented by twist structures; this time the conucleus will be on ${\bf{Tw}(L,\imath)}$.
Given a residuated lattice
${\bf L}$ and an element $\imath\in L$, we define the function $$\tw_\text{Tw}(a,b)=(a,  \imath \rd a \mt a \ld \imath)$$
on  ${\bf{Tw}(L,\imath)}$.
Recall that $\imath$ is  \textbf{cyclic} if $x \ld \imath= \imath \rd x$, for all $x$. In that case $ab \jn ba \leq \imath$  iff $a b \leq \imath$ iff $ba \leq \imath$, for all $a,b$. So, if $\imath$ is a cyclic (and in particular in the commutative case), $
\text{Tw}(L,\imath) = \{(a,b)\in L\times L^\partial: a b\leq \imath\}$
and also $\tw_\text{Tw}(a,b)=(a, a \ld  \imath )$.

Note that, under the assumption that $\imath$ is cyclic,  $\tw_\text{Tw}$ satisfies (C\ref{Cord1}),  (C\ref{Cidem}), (C\ref{Cmono}) and (C\ref{Cneutral}). Moreover, we can prove:

\begin{lem}
    Assume that the function $\tw_\text{Tw}(a,b)=(a,  \imath \rd a \mt a \ld \imath)$ is defined on ${\bf{Tw}(L,\imath)}$ with $\imath $ cyclic. Then for each $a, b, c, d\in L$ we have:
    \begin{enumerate}
        \item $\tw_\text{Tw}((a,b)\cdot  (c,d))=\tw_\text{Tw}(a,b)\cdot\tw_\text{Tw}(c,d)$
        \item $\tw_\text{Tw}((a,b)\vee (c,d))= \tw_\text{Tw}(a,b)\vee \tw_\text{Tw}(a,b)$
        \item $\left( \tw_\text{Tw}(a,b)\cdot (c,d)\right)\vee \left((a,b) \cdot \tw_\text{Tw}(c,d)\right)= (a,b)\cdot (c,d).$
    \end{enumerate}
    In particular, $\tw_\text{Tw}(a,b)$ is a conucleus.
\end{lem}

\begin{proof}
Using Lemma 2.6(6) of \cite{GJKO}, we have that
\begin{align*}\tw_\text{Tw}((a,b)\cdot (c,d)) &= (ac,(ac)\ld\imath) = (ac,\imath\rd(ac) \wedge (ac)\ld\imath )\\
&= (ac,  (c\ld\imath)\rd a)\wedge (c\ld(a\ld\imath)) = (a, a \ld \imath)\cdot (c, c\ld\imath)\\
&=\tw_\text{Tw}(a,b)\cdot\tw_\text{Tw}(c,d).  \end{align*}
    Using  Lemma 2.6(3) of \cite{GJKO}, we also have that:
 \begin{align*}
  \tw_\text{Tw}((a,b)\vee (c,d)) &=  \tw_\text{Tw}(a\vee c, b \mt d) = (a\vee c, (a\vee c)\ld \imath)\\
  &= (a\vee c, (a\ld \imath) \wedge (c\ld \imath)) = (a, a \ld \imath) \jn (c, c\ld\imath)\\
  &= \tw_\text{Tw}(a,b)\vee \tw_\text{Tw}(a,b).\end{align*}
Finally we observe that
 \begin{align*}
   \left( \tw_\text{Tw}(a,b) (c,d)\right)\vee \left((a,b)  \tw_\text{Tw}(c,d)\right) &=
    (a, a \ld \imath)(c,d)\jn (a,b)(c, c\ld\imath) \\
  &\hspace{-1cm}=  (ac, (d\rd a\wedge c\ld(a\ld\imath)) ) \jn (ac, (c\ld b\wedge (c\ld\imath)\rd a)) \\
    &\hspace{-1cm}=  (ac, (d\rd a\wedge c\ld(a\ld\imath))\wedge (c\ld b\wedge (c\ld\imath)\rd a)) \\
  &\hspace{-1cm}=  (ac, (d\rd a\wedge ac \ld \imath)\wedge (c\ld b\wedge ac\ld\imath)) \\
  &\hspace{-1cm}=(ac, d\rd a\wedge c\ld b) = (a,b)\cdot (c,d),
\end{align*}
where we used that $ac(d\rd a\wedge c\ld b) \leq ac (c \ld b)\leq ab \leq \imath$, so $ d\rd a\wedge c\ld b \leq ac \ld \imath$.
\end{proof}

\section{Nelson conucleus algebras}

Motivated by the properties of the pair of the involutive residuated lattice ${\bf {Tw}(L,\imath)}$ and the conucleus $\tw_\text{Tw}$,  we define the main class of algebras of the paper and link them to the preceding constructions.

\subsection{Nelson conuclei and the variety $\mathcal{NCA}$}

An operator $\tw$ on a residuated lattice ${\bf A}$ is called a  \textbf{Nelson conucleus} if $\tw$ is a conucleus and it also satisfies:
\begin{enumerate}[{(T}1{)}]
    \item\label{Tsup} $\tw(x\vee y) = \tw(x)\vee \tw(y)$
    \item\label{Tprod} $\tw(x y) = \tw(x) \tw(y)$
    \item\label{Tmult}  $x y \leq \tw(x) y \jn x\tw(y)$
\end{enumerate}

Any conucleus satisfies
$x  y \ge \tw(x) y \jn x\tw(y)$.  Then we have:

\begin{lem} Let $\m A$ be  a residuated lattice and $\tw$ a conucleus on it satisfying (T\ref{Tsup}) and (T\ref{Tprod}). Then $\tw$ is a Nelson conucleus iff
    \begin{enumerate}[{(T}1{)}]
        \setcounter{enumi}{3}
        \item\label{Tmulteq} $x  y = \tw(x) y \jn x\tw(y).$
    \end{enumerate}
\end{lem}
Observe that a Nelson conucleus  $\tw$ also satisfies $\tw(e)=e$
since  $e=ee=\tw(e)e \jn e\tw(e)=\tw(e)$. Thus given a residuated lattice ${\bf A}$ and a Nelson conucleus $\tw$ on ${\bf   A}$, the conucleus image ${\bf A}_{\tw}=(\tw[{\bf A}], \mt_\tw, \jn , \cdot, \ld_\tw, \rd_\tw, e)$ is a residuated lattice.

From the results of Section \ref{section_natural_conuclei} we get:

\begin{lem}\label{lem:example_of_conucleus} Given a residuated lattice ${\bf L}$ and a cyclic element $\imath \in L$, the operator $\tw_{\rm Tw}$ defined on ${\bf Tw}({\bf L},\imath)$ by
$$\tw_\text{Tw}(a,b)=(a, a \ld \imath)$$
 is a Nelson conucleus.
\end{lem}

If ${\bf L}$ an integral and commutative residuated lattice, then ${\bf Tw}({\bf L},1)$ is a Kalman lattice and $\tw_\text{Tw}(a,b)=(a,1).$ If ${\bf H}$ is a Heyting algebra and ${\bf Tw}({\bf H},0)$ is a Nelson residuated lattice, then
$\tw_\text{Tw}(a,b)=(a,a \to 0)$.

We consider in Figure \ref{23_element_twists}
 all the twist structures from Figures \ref{2_element_twists} and \ref{3_element_twists} associated with all the two or three-element residuated lattices, together with the sets $\tw[{\bf A}]$.

\begin{figure}[!htp]
    \centering
    \begin{tikzpicture}[scale=0.7]
    \draw[black, fill=black] (0-6,-6+0) circle (0.07 cm);
    \node[circle,gray,line width=1pt,draw]  at (0-6,-6+0) {};
    \draw[black, fill=black] (-2-6,-6+0) circle (0.07 cm);
    \draw[black] (2-6,-6+0) circle (0.07 cm);
    \draw[black, fill=black] (-1-6,-6+1) circle (0.07 cm);
    \draw[black] (1-6,-6+1) circle (0.07 cm);
    \draw[black, fill=black] (-1-6,-6+-1) circle (0.07 cm);
    \draw[black] (1-6,-6+-1) circle (0.07 cm);
    \draw[black, fill=black] (0-6,-6+2) +(-2pt,-2pt) rectangle +(2pt,2pt) ;
    \node[circle,gray,line width=1pt,draw]  at (0-6,-6+2) {};
    \draw[black, fill=black] (0-6,-6+-2) circle (0.07 cm);
    \node[circle,gray,line width=1pt,draw]  at (0-6,-6+-2) {};
    \draw[black] (0-6,-6+0)--(-1-6,-6+1);
    \draw[dotted] (0-6,-6+0)--(1-6,-6+1);
    \draw[black] (0-6,-6+0)--(-1-6,-6+-1);
    \draw[dotted] (0-6,-6+0)--(1-6,-6+-1);
    \draw[black] (-2-6,-6+0)--(-1-6,-6+-1);
    \draw[black] (-2-6,-6+0)--(-1-6,-6+1);
    \draw[dotted] (2-6,-6+0)--(1-6,-6+-1);
    \draw[dotted] (2-6,-6+0)--(1-6,-6+1);
    \draw[black] (0-6,-6+-2)--(-1-6,-6+-1);
    \draw[dotted] (0-6,-6+-2)--(1-6,-6+-1);
    \draw[black] (0-6,-6+2)--(-1-6,-6+1);
    \draw[dotted] (0-6,-6+2)--(1-6,-6+1);
    \node[right] at (0-6,-6+0) {\tiny$(a,a)$};
    \node[right] at (-2-6,-6+0) {\tiny$(0,0)$};
    \node[right] at (-1-6,-6+1) {\tiny$(a,0)$};
    \node[right] at (-1-6,-6+-1) {\tiny$(0,a)$};
    \node[right] at (0-6,-6+2) {{\tiny$(1,0)$}};
    \node[right] at (0-6,-6+-2) {\tiny$(0,1)$};
    \node at (0-6,-6+-3) {$\text{Tw}(\text{\L}_3,0)$};
    \draw[black, fill=black] (0-6,0) circle (0.07 cm);
    \draw[black, fill=black] (-2-6,0) circle (0.07 cm);
    \draw[black] (2-6,0) circle (0.07 cm);
    \draw[black, fill=black] (-1-6,1) circle (0.07 cm);
    \draw[black, fill=black] (1-6,1) +(-2pt,-2pt) rectangle +(2pt,2pt) ;
    \node[circle,gray,line width=1pt,draw]  at (1-6,1) {};
    \draw[black, fill=black] (-1-6,-1) circle (0.07 cm);
    \draw[black, fill=black] (1-6,-1) circle (0.07 cm);
    \node[circle,gray,line width=1pt,draw]  at (1-6,-1) {};
    \draw[black, fill=black] (0-6,2) circle (0.07 cm);
    \draw[black, fill=black] (0-6,-2) circle (0.07 cm);
    \node[circle,gray,line width=1pt,draw]  at (0-6,-2) {};
    \draw[black] (0-6,0)--(-1-6,1);
    \draw[black] (0-6,0)--(1-6,1);
    \draw[black] (0-6,0)--(-1-6,-1);
    \draw[black] (0-6,0)--(1-6,-1);
    \draw[black] (-2-6,0)--(-1-6,-1);
    \draw[black] (-2-6,0)--(-1-6,1);
    \draw[dotted] (2-6,0)--(1-6,-1);
    \draw[dotted] (2-6,0)--(1-6,1);
    \draw[black] (0-6,-2)--(-1-6,-1);
    \draw[black] (0-6,-2)--(1-6,-1);
    \draw[black] (0-6,2)--(-1-6,1);
    \draw[black] (0-6,2)--(1-6,1);
    \node[right] at (0-6,0) {\tiny$(a,a)$};
    \node[right] at (-2-6,0) {\tiny$(0,0)$};
    \node[right] at (2-6,0) {\tiny$(1,1)$};
    \node[right] at (-1-6,1) {\tiny$(a,0)$};
    \node[right] at (1-6,1) {{\tiny$(1,a)$}};
    \node[right] at (-1-6,-1) {\tiny$(0,a)$};
    \node[right] at (1-6,-1) {\tiny$(a,1)$};
    \node[right] at (0-6,2) {\tiny$(1,0)$};
    \node[right] at (0-6,-2) {\tiny$(0,1)$};
    \node at (0-6,-3) {$\text{Tw}(\text{\L}_3,a)$};
    \draw[black, fill=black] (0-6,6+0) circle (0.07 cm);
    \draw[black, fill=black] (-2-6,6+0) circle (0.07 cm);
    \draw[black, fill=black] (2-6,6+0) +(-2pt,-2pt) rectangle +(2pt,2pt) ;
    \node[circle,gray,line width=1pt,draw]  at (2-6,6+0) {};
    \draw[black, fill=black] (-1-6,6+1) circle (0.07 cm);
    \draw[black, fill=black] (1-6,6+1) circle (0.07 cm);
    \draw[black, fill=black] (-1-6,6+-1) circle (0.07 cm);
    \draw[black, fill=black] (1-6,6+-1) circle (0.07 cm);
    \node[circle,gray,line width=1pt,draw]  at (1-6,6+-1) {};
    \draw[black, fill=black] (0-6,6+2) circle (0.07 cm);
    \draw[black, fill=black] (0-6,6+-2) circle (0.07 cm);
    \node[circle,gray,line width=1pt,draw]  at (0-6,6+-2) {};
    \draw[black] (0-6,6+0)--(-1-6,6+1);
    \draw[black] (0-6,6+0)--(1-6,6+1);
    \draw[black] (0-6,6+0)--(-1-6,6+-1);
    \draw[black] (0-6,6+0)--(1-6,6+-1);
    \draw[black] (-2-6,6+0)--(-1-6,6+-1);
    \draw[black] (-2-6,6+0)--(-1-6,6+1);
    \draw[black] (2-6,6+0)--(1-6,6+-1);
    \draw[black] (2-6,6+0)--(1-6,6+1);
    \draw[black] (0-6,6+-2)--(-1-6,6+-1);
    \draw[black] (0-6,6+-2)--(1-6,6+-1);
    \draw[black] (0-6,6+2)--(-1-6,6+1);
    \draw[black] (0-6,6+2)--(1-6,6+1);
    \node[right] at (0-6,6+0) {\tiny$(a,a)$};
    \node[right] at (-2-6,6+0) {\tiny$(0,0)$};
    \node[right] at (2-6,6+0) {{\tiny$(1,1)$}};
    \node[right] at (-1-6,6+1) {\tiny$(a,0)$};
    \node[right] at (1-6,6+1) {\tiny$(1,a)$};
    \node[right] at (-1-6,6+-1) {\tiny$(0,a)$};
    \node[right] at (1-6,6+-1) {\tiny$(a,1)$};
    \node[right] at (0-6,6+2) {\tiny$(1,0)$};
    \node[right] at (0-6,6+-2) {\tiny$(0,1)$};
    \node at (0-6,6+-3) {$\text{Tw}(\text{\L}_3,1)$};
    \draw[black, fill=black] (-6,12+0) circle (0.07 cm);
    \draw[black, fill=black] (-6,12+1) circle (0.07 cm);
    \draw[black, fill=black] (-6,12-1) circle (0.07 cm);
    \draw[black] (-6,12+0)--(-6,12+1);
    \draw[black] (-6,12+0)--(-6,12-1);
    \node[right] at (-6,12+0) {\tiny$a$};
    \node[right] at (-6,12+1) {\tiny$1$};
    \node[right] at (-6,12-1) {\tiny$0=a^2$};
    \node at (-6,12-2) {$\text{\L}_3$};
    \draw[black] (0,-6+0) circle (0.07 cm);
    \draw[black, fill=black] (-2,-6+0) circle (0.07 cm);
    \draw[black] (2,-6+0) circle (0.07 cm);
    \draw[black, fill=black] (-1,-6+1) circle (0.07 cm);
    \node[circle,gray,line width=1pt,draw]  at (-1,-6+1) {};
    \draw[black] (1,-6+1) circle (0.07 cm);
    \draw[black, fill=black] (-1,-6+-1) circle (0.07 cm);
    \draw[black] (1,-6+-1) circle (0.07 cm);
    \draw[black, fill=black] (0,-6+2) +(-2pt,-2pt) rectangle +(2pt,2pt) ;
    \node[circle,gray,line width=1pt,draw]  at (0,-6+2) {};
    \draw[black, fill=black] (0,-6+-2) circle (0.07 cm);
    \node[circle,gray,line width=1pt,draw]  at (0,-6+-2) {};
    \draw[dotted] (0,-6+0)--(-1,-6+1);
    \draw[dotted] (0,-6+0)--(1,-6+1);
    \draw[dotted] (0,-6+0)--(-1,-6+-1);
    \draw[dotted] (0,-6+0)--(1,-6+-1);
    \draw[black] (-2,-6+0)--(-1,-6+-1);
    \draw[black] (-2,-6+0)--(-1,-6+1);
    \draw[dotted] (2,-6+0)--(1,-6+-1);
    \draw[dotted] (2,-6+0)--(1,-6+1);
    \draw[black] (0,-6+-2)--(-1,-6+-1);
    \draw[dotted] (0,-6+-2)--(1,-6+-1);
    \draw[black] (0,-6+2)--(-1,-6+1);
    \draw[dotted] (0,-6+2)--(1,-6+1);
    \node[right] at (-2,-6+0) {\tiny$(0,0)$};
    \node[right] at (-1,-6+1) {\tiny$(a,0)$};
    \node[right] at (-1,-6+-1) {\tiny$(0,a)$};
    \node[right] at (0,-6+2) {{\tiny$(1,0)$}};
    \node[right] at (0,-6+-2) {\tiny$(0,1)$};
    \node at (0,-6+-3) {$\text{Tw}(G_3,0)$};
    \draw[black, fill=black] (0,0) circle (0.07 cm);
    \draw[black, fill=black] (-2,0) circle (0.07 cm);
    \draw[black] (2,0) circle (0.07 cm);
    \draw[black, fill=black] (-1,1) circle (0.07 cm);
    \draw[black, fill=black] (1,1) +(-2pt,-2pt) rectangle +(2pt,2pt) ;
    \node[circle,gray,line width=1pt,draw]  at (1,1) {};
    \draw[black, fill=black] (-1,-1) circle (0.07 cm);
    \draw[black, fill=black] (1,-1) circle (0.07 cm);
    \node[circle,gray,line width=1pt,draw]  at (1,-1) {};
    \draw[black, fill=black] (0,2) circle (0.07 cm);
    \draw[black, fill=black] (0,-2) circle (0.07 cm);
    \node[circle,gray,line width=1pt,draw]  at (0,-2) {};
    \draw[black] (0,0)--(-1,1);
    \draw[black] (0,0)--(1,1);
    \draw[black] (0,0)--(-1,-1);
    \draw[black] (0,0)--(1,-1);
    \draw[black] (-2,0)--(-1,-1);
    \draw[black] (-2,0)--(-1,1);
    \draw[dotted] (2,0)--(1,-1);
    \draw[dotted] (2,0)--(1,1);
    \draw[black] (0,-2)--(-1,-1);
    \draw[black] (0,-2)--(1,-1);
    \draw[black] (0,2)--(-1,1);
    \draw[black] (0,2)--(1,1);
    \node[right] at (0,0) {\tiny$(a,a)$};
    \node[right] at (-2,0) {\tiny$(0,0)$};
    \node[right] at (-1,1) {\tiny$(a,0)$};
    \node[right] at (1,1) {{\tiny$(1,a)$}};
    \node[right] at (-1,-1) {\tiny$(0,a)$};
    \node[right] at (1,-1) {\tiny$(a,1)$};
    \node[right] at (0,2) {\tiny$(1,0)$};
    \node[right] at (0,-2) {\tiny$(0,1)$};
    \node at (0,-3) {$\text{Tw}(G_3,a)$};
    \draw[black, fill=black] (0,6+0) circle (0.07 cm);
    \draw[black, fill=black] (-2,6+0) circle (0.07 cm);
    \draw[black, fill=black] (2,6+0) +(-2pt,-2pt) rectangle +(2pt,2pt) ;
    \node[circle,gray,line width=1pt,draw]  at (2,6+0) {};
    \draw[black, fill=black] (-1,6+1) circle (0.07 cm);
    \draw[black, fill=black] (1,6+1) circle (0.07 cm);
    \draw[black, fill=black] (-1,6+-1) circle (0.07 cm);
    \draw[black, fill=black] (1,6+-1) circle (0.07 cm);
    \node[circle,gray,line width=1pt,draw]  at (1,6+-1) {};
    \draw[black, fill=black] (0,6+2) circle (0.07 cm);
    \draw[black, fill=black] (0,6+-2) circle (0.07 cm);
    \node[circle,gray,line width=1pt,draw]  at (0,6+-2) {};
    \draw[black] (0,6+0)--(-1,6+1);
    \draw[black] (0,6+0)--(1,6+1);
    \draw[black] (0,6+0)--(-1,6+-1);
    \draw[black] (0,6+0)--(1,6+-1);
    \draw[black] (-2,6+0)--(-1,6+-1);
    \draw[black] (-2,6+0)--(-1,6+1);
    \draw[black] (2,6+0)--(1,6+-1);
    \draw[black] (2,6+0)--(1,6+1);
    \draw[black] (0,6+-2)--(-1,6+-1);
    \draw[black] (0,6+-2)--(1,6+-1);
    \draw[black] (0,6+2)--(-1,6+1);
    \draw[black] (0,6+2)--(1,6+1);
    \node[right] at (0,6+0) {\tiny$(a,a)$};
    \node[right] at (-2,6+0) {\tiny$(0,0)$};
    \node[right] at (2,6+0) {{\tiny$(1,1)$}};
    \node[right] at (-1,6+1) {\tiny$(a,0)$};
    \node[right] at (1,6+1) {\tiny$(1,a)$};
    \node[right] at (-1,6+-1) {\tiny$(0,a)$};
    \node[right] at (1,6+-1) {\tiny$(a,1)$};
    \node[right] at (0,6+2) {\tiny$(1,0)$};
    \node[right] at (0,6+-2) {\tiny$(0,1)$};
    \node at (0,6+-3) {$\text{Tw}(G_3,1)$};
    \draw[black, fill=black] (0,12+0) circle (0.07 cm);
    \draw[black, fill=black] (0,12+1) circle (0.07 cm);
    \draw[black, fill=black] (0,12-1) circle (0.07 cm);
    \draw[black] (0,12+0)--(0,12+1);
    \draw[black] (0,12+0)--(0,12-1);
    \node[right] at (0,12+0) {\tiny$a=a^2$};
    \node[right] at (0,12+1) {\tiny$1$};
    \node[right] at (0,12-1) {\tiny$0$};
    \node at (0,12-2) {$G_3$};
    \draw[black] (0+6,-6+0) circle (0.07 cm);
    \draw[black, fill=black] (-2+6,-6+0) circle (0.07 cm);
    \draw[black] (2+6,-6+0) circle (0.07 cm);
    \draw[black, fill=black] (-1+6,-6+1) +(-2pt,-2pt) rectangle +(2pt,2pt) ;
    \node[circle,gray,line width=1pt,draw]  at (-1+6,-6+1) {};
    \draw[black] (1+6,-6+1) circle (0.07 cm);
    \draw[black, fill=black] (-1+6,-6+-1) circle (0.07 cm);
    \draw[black] (1+6,-6+-1) circle (0.07 cm);
    \draw[black, fill=black] (0+6,-6+2) circle (0.07 cm);
    \node[circle,gray,line width=1pt,draw]  at (0+6,-6+2) {};
    \draw[black, fill=black] (0+6,-6+-2) circle (0.07 cm);
    \node[circle,gray,line width=1pt,draw]  at (0+6,-6+-2) {};
    \draw[dotted] (0+6,-6+0)--(-1+6,-6+1);
    \draw[dotted] (0+6,-6+0)--(1+6,-6+1);
    \draw[dotted] (0+6,-6+0)--(-1+6,-6+-1);
    \draw[dotted] (0+6,-6+0)--(1+6,-6+-1);
    \draw[black] (-2+6,-6+0)--(-1+6,-6+-1);
    \draw[black] (-2+6,-6+0)--(-1+6,-6+1);
    \draw[dotted] (2+6,-6+0)--(1+6,-6+-1);
    \draw[dotted] (2+6,-6+0)--(1+6,-6+1);
    \draw[black] (0+6,-6+-2)--(-1+6,-6+-1);
    \draw[dotted] (0+6,-6+-2)--(1+6,-6+-1);
    \draw[black] (0+6,-6+2)--(-1+6,-6+1);
    \draw[dotted] (0+6,-6+2)--(1+6,-6+1);
    \node[right] at (-2+6,-6+0) {\tiny$(0,0)$};
    \node[right] at (-1+6,-6+1) {{\tiny$(e,0)$}};
    \node[right] at (-1+6,-6+-1) {\tiny$(0,e)$};
    \node[right] at (0+6,-6+2) {\tiny$(\top,0)$};
    \node[right] at (0+6,-6+-2) {\tiny$(0,\top)$};
    \node at (0+6,-6+-3) {$\text{Tw}(S_3,0)$};
    \draw[black, fill=black] (0+6,0) +(-2pt,-2pt) rectangle +(2pt,2pt) ;
    \node[circle,gray,line width=1pt,draw]  at (0+6,0) {};
    \draw[black, fill=black] (-2+6,0) circle (0.07 cm);
    \draw[black] (2+6,0) circle (0.07 cm);
    \draw[black, fill=black] (-1+6,1) circle (0.07 cm);
    \draw[black] (1+6,1) circle (0.07 cm);
    \draw[black, fill=black] (-1+6,-1) circle (0.07 cm);
    \draw[black] (1+6,-1) circle (0.07 cm);
    \draw[black, fill=black] (0+6,2) circle (0.07 cm);
    \node[circle,gray,line width=1pt,draw]  at (0+6,2) {};
    \draw[black, fill=black] (0+6,-2) circle (0.07 cm);
    \node[circle,gray,line width=1pt,draw]  at (0+6,-2) {};
    \draw[black] (0+6,0)--(-1+6,1);
    \draw[dotted] (0+6,0)--(1+6,1);
    \draw[black] (0+6,0)--(-1+6,-1);
    \draw[dotted] (0+6,0)--(1+6,-1);
    \draw[black] (-2+6,0)--(-1+6,-1);
    \draw[black] (-2+6,0)--(-1+6,1);
    \draw[dotted] (2+6,0)--(1+6,-1);
    \draw[dotted] (2+6,0)--(1+6,1);
    \draw[black] (0+6,-2)--(-1+6,-1);
    \draw[dotted] (0+6,-2)--(1+6,-1);
    \draw[black] (0+6,2)--(-1+6,1);
    \draw[dotted] (0+6,2)--(1+6,1);
    \node[right] at (0+6,0) {{\tiny$(e,e)$}};
    \node[right] at (-2+6,0) {\tiny$(0,0)$};
    \node[right] at (-1+6,1) {\tiny$(e,0)$};
    \node[right] at (-1+6,-1) {\tiny$(0,e)$};
    \node[right] at (0+6,2) {\tiny$(\top,0)$};
    \node[right] at (0+6,-2) {\tiny$(0,\top)$};
    \node at (0+6,-3) {$\text{Tw}(S_3,e)$};
    \draw[black, fill=black] (0+6,6+0) circle (0.07 cm);
    \draw[black, fill=black] (-2+6,6+0) circle (0.07 cm);
    \draw[black, fill=black] (2+6,6+0) circle (0.07 cm);
    \node[circle,gray,line width=1pt,draw]  at (2+6,6+0) {};
    \draw[black, fill=black] (-1+6,6+1) circle (0.07 cm);
    \draw[black, fill=black] (1+6,6+1) circle (0.07 cm);
    \draw[black, fill=black] (-1+6,6+-1) circle (0.07 cm);
    \draw[black, fill=black] (1+6,6+-1) +(-2pt,-2pt) rectangle +(2pt,2pt) ;
    \node[circle,gray,line width=1pt,draw]  at (1+6,6+-1) {};
    \draw[black, fill=black] (0+6,6+2) circle (0.07 cm);
    \draw[black, fill=black] (0+6,6+-2) circle (0.07 cm);
    \node[circle,gray,line width=1pt,draw]  at (0+6,6+-2) {};
    \draw[black] (0+6,6+0)--(-1+6,6+1);
    \draw[black] (0+6,6+0)--(1+6,6+1);
    \draw[black] (0+6,6+0)--(-1+6,6+-1);
    \draw[black] (0+6,6+0)--(1+6,6+-1);
    \draw[black] (-2+6,6+0)--(-1+6,6+-1);
    \draw[black] (-2+6,6+0)--(-1+6,6+1);
    \draw[black] (2+6,6+0)--(1+6,6+-1);
    \draw[black] (2+6,6+0)--(1+6,6+1);
    \draw[black] (0+6,6+-2)--(-1+6,6+-1);
    \draw[black] (0+6,6+-2)--(1+6,6+-1);
    \draw[black] (0+6,6+2)--(-1+6,6+1);
    \draw[black] (0+6,6+2)--(1+6,6+1);
    \node[right] at (0+6,6+0) {\tiny$(e,e)$};
    \node[right] at (-2+6,6+0) {\tiny$(0,0)$};
    \node[right] at (2+6,6+0) {\tiny$(\top,\top)$};
    \node[right] at (-1+6,6+1) {\tiny$(e,0)$};
    \node[right] at (1+6,6+1) {\tiny$(\top,e)$};
    \node[right] at (-1+6,6+-1) {\tiny$(0,e)$};
    \node[right] at (1+6,6+-1) {{\tiny$(e,\top)$}};
    \node[right] at (0+6,6+2) {\tiny$(\top,0)$};
    \node[right] at (0+6,6+-2) {\tiny$(0,\top)$};
    \node at (0+6,6+-3) {$\text{Tw}(S_3,\top)$};
    \draw[black, fill=black] (+6,12+0) circle (0.07 cm);
    \draw[black, fill=black] (+6,12+1) circle (0.07 cm);
    \draw[black, fill=black] (+6,12-1) circle (0.07 cm);
    \draw[black] (+6,12+0)--(+6,12+1);
    \draw[black] (+6,12+0)--(+6,12-1);
    \node[right] at (+6,12+0) {\tiny$e$};
    \node[right] at (+6,12+1) {\tiny$\top$};
    \node[right] at (+6,12-1) {\tiny$0$};
    \node at (+6,12-2) {$S_3$};
    \draw[black, fill=black] (0-12,0) +(-2pt,-2pt) rectangle +(2pt,2pt) ;
    \node[circle,gray,line width=1pt,draw]  at (0-12,0) {};
    \draw[black, fill=black] (-1-12,-1) circle (0.07 cm);
    \draw[black] (1-12,-1) circle (0.07 cm);
    \draw[black, fill=black] (0-12,-2) circle (0.07 cm);
    \node[circle,gray,line width=1pt,draw]  at (0-12,-2) {};
    \draw[black] (0-12,0)--(-1-12,-1);
    \draw[dotted] (0-12,0)--(1-12,-1);
    \draw[black] (0-12,-2)--(-1-12,-1);
    \draw[dotted] (0-12,-2)--(1-12,-1);
    \node[right] at (0-12,0) {{\tiny$(1,0)$}};
    \node[right] at (-1-12,-1) {\tiny$(0,0)$};
    \node[right] at (0-12,-2) {\tiny$(0,1)$};
    \node at (0-12,-3) {$\text{Tw}({\bf 2},0)$};
    \draw[black, fill=black] (0-12,6+0) circle (0.07 cm);
    \draw[black, fill=black] (-1-12,6+-1) circle (0.07 cm);
    \draw[black, fill=black] (1-12,6+-1) +(-2pt,-2pt) rectangle +(2pt,2pt) ;
    \node[circle,gray,line width=1pt,draw]  at (1-12,6+-1) {};
    \draw[black, fill=black] (0-12,6+-2) circle (0.07 cm);
    \node[circle,gray,line width=1pt,draw]  at (0-12,6+-2) {};
    \draw[black] (0-12,6+0)--(-1-12,6+-1);
    \draw[black] (0-12,6+0)--(1-12,6+-1);
    \draw[black] (0-12,6+-2)--(-1-12,6+-1);
    \draw[black] (0-12,6+-2)--(1-12,6+-1);
    \node[right] at (0-12,6+0) {\tiny$(1,0)$};
    \node[right] at (-1-12,6+-1) {\tiny$(0,0)$};
    \node[right] at (1-12,6+-1) {{\tiny$(1,1)$}};
    \node[right] at (0-12,6+-2) {\tiny$(0,1)$};
    \node at (0-12,6+-3) {$\text{Tw}({\bf 2},1)$};
    \draw[black, fill=black] (-12,12+0) circle (0.07 cm);
    \draw[black, fill=black] (-12,12-1) circle (0.07 cm);
    \draw[black] (-12,12+0)--(-12,12-1);
    \node[right] at (-12,12+0) {\tiny$1$};
    \node[right] at (-12,12-1) {\tiny$0$};
    \node at (-12,12-2) {$\bf 2$};
    \end{tikzpicture}
    \caption{The 2 or 3-element residuated lattices $\bf 2$, $\bf{\text{\L}_3}$, $\bf G_3$ and $\bf S_3$, together with the lattice structure of all their possible twist structures, and the image of the conucleus. In each case, the images of $\tw$ are shown in gray, and the identity element for the product $(1,\imath)$ is marked as a square.}
    \label{23_element_twists}
\end{figure}

\medskip

In the involutive case (T\ref{Tmulteq}) can be rephrased:

\begin{lem}\label{lem:divisions}
    Let $\m A$ be  an involutive residuated lattice and $\tw$ a conucleus on $\m A.$ The operator $\tw$ satisfies (T\ref{Tmulteq}) iff it satisfies one of the following equivalent identities
    \begin{enumerate}[{(T}1{)}]
        \setcounter{enumi}{4}
        \item\label{TNelson}  $x\ld y=(\tw(x)\ld y)\wedge (\ln x\rd\tw(\ln y)),\ \ $
        $\ x\ld y=(\tw(x)\ld y)\wedge (x\ld \ln\tw(\ln y)),$
        \item[] $y \rd x=( y \rd \tw(x))\wedge (\tw(\ln y)\ld \ln x), \ \ $
        $\ \ y \rd x=( y \rd \tw(x))\wedge (\ln \tw(\ln y)\rd x).$
    \end{enumerate}
    In this case we also have that $\tw(\ln e)$ is cyclic in $\m A_\tw$.

\end{lem}

\begin{proof}
    For (T\ref{TNelson}), by (T\ref{Tmulteq}) and Lemma~\ref{lem:propiedades_involucion}(1), we have
    \begin{align*}
    (x\ld y) &= \ln(\ln y\cdot x) =\ln(\tw(\ln y)\cdot x\vee (\ln y)\cdot \tw(x))\\
    &=\ln (\tw(\ln y)\cdot x) \wedge \ln((\ln y)\cdot \tw(x))=(\ln x\rd \tw(\ln y)) \wedge (\tw(x)\ld y).
    \end{align*}
    The converse follows the same reasoning and analogously we get the other equalities, using that $a \ld b= \ln a \rd \ln b$.
    Finally, for cyclicity,
    recalling (C\ref{Cimpl})  we obtain
    \begin{align*}
    \tw(x)\ld_\tw \tw(\ln e) &=\tw(\tw(x)\ld \tw(\ln e)) =\tw(\tw(x)\ld \ln e)\\
    &=\tw(\ln \tw(x)) =\tw(\ln e \rd \tw(x)) =\tw(\tw(\ln e)\rd \tw(x)) = \tw(\ln e) \rd_\tw \tw(x).\end{align*}
\end{proof}

Note that  $ \tw(x)\leq \tw(y)$ iff $ \tw(x)\leq y$ holds for all interior operators.

\begin{lem}\label{quasiequation} Let ${\bf A}$ be a residuated lattice,  $\tw$ a Nelson conucleus in ${\bf A}$ and $x,y\in A$.

    If $\tw(x)=\tw(y)$ and $\tw(\ln x)=\tw(\ln y)$, then $x=y$.
\end{lem}
\begin{proof}
    Observe that $\tw(x)\le \tw(y)$ iff $\tw(x)\le y.$ Now  $\tw(x)\leq y$ and $\tw(\ln y)\leq \ln x$, iff $e\leq \tw(x)\ld y$ and $e\leq \ln x \rd \tw(\ln y)$, iff $e\leq (\tw(x)\ld y)\wedge (\ln x \rd \tw(\ln y))$. By Equation (T\ref{TNelson}) this is equivalent to $e \le x \ld y$ and to $x\leq y$.
\end{proof}

We define the variety $\mathcal{NCA}$ of \textbf{Nelson conucleus algebras} whose elements are  algebras $({\bf A},\tw)$ such that ${\bf A}$ is a (cyclic) involutive residuated lattice and $\tw$  is a Nelson conucleus on ${\bf A}$.
As an immediate consequence of Lemma \ref{lem:example_of_conucleus} we have:

\begin{thm}
Given a residuated lattice ${\bf L}$ and a cyclic element $\imath\in L$, the pair $({\bf Tw(L,\imath), \tw_\text{Tw})}$, where $\tw_\text{Tw}(a,b)=(a, \imath \rd a \mt a \ld \imath)$ as before, is in $\mathcal{NCA}$.
\end{thm}

If $({\bf A},\tw) \in \mathcal{NCA}$, we define
$$x \dld  y= \tw(x) \ld y \ \mbox{ and  } \ \  y \drd x =  y\rd \tw(x).$$
With this notation, from   (C\ref{Cimpl}) we get $\tw(x \dld y)=\tw(\tw(x) \ld y)=\tw(\tw(x) \ld \tw(y))=\tw(x) \ld_{\tw} \tw(y),$ $ \ \tw(x \drd y)=\tw(\tw(x) \rd y)=\tw(\tw(x) \rd \tw(y))=\tw(x) \rd_{\tw} \tw(y)$
and   the equations (T\ref{TNelson})   become

\begin{equation}\label{eq:nelson'}\tag{T\ref{TNelson}'}
x \ld y=(x\dld y)\wedge (\ln x\drd \ln y) \qquad
y \rd x=(y\drd x)\wedge (\ln y\dld \ln x).
\end{equation}

\subsection{Twist-representation and a categorical adjunction}
So far we have described the process where from a residuated lattice $\m L$ and a cyclic element $\imath \in L$, we construct an algebra  $({\bf Tw}(L,\imath), \tw_\text{Tw})$ in $\mathcal{NCA}$.

We consider the category $\mathcal{RL}_\text{cy}$ with objects algebras $(\m L, \imath)$, where  $\m L$ is a residuated lattice and $\imath$ is a cyclic element of $\m L$; the morphisms are homomorphisms of these algebras (they preserve the cyclic element).
Also, note that $\mathcal{NCA}$ defines a category where the morphisms are the algebraic homomorphisms.

\medskip

For an object $(\m L, \imath) \in \mathcal{RL}_\text{cy}$ and a morphism $f$ in $\mathcal{RL}_\text{cy}$, we define
$$\m T(\m L, \imath)  = ({\bf Tw}(L,\imath), \tw_\text{Tw}) \ \ \ \mbox{ and } \ \ \ \m {T}(f)(a,b)  =(f(a), f(b)).$$

It can be easily verified that $\m{T}$ is a functor from $\mathcal{RL}_\text{cy}$ to $\mathcal{NCA}$.

\medskip

We also have the reverse process: given an algebra $({\bf A}, \tw)\in \mathcal{NCA}$, the algebra
\begin{align*}{\bf A}_{\tw}=  (\tw[A],\vee,\wedge_{\tw},\cdot, \ld_{\tw}, \rd_{\tw}, e)\end{align*}
is a  residuated lattice and $\tw(\ln e)$ cyclic element of ${\bf A}_{\tw}$. Therefore, for an object $({\bf A}, \tw)$ and a homomorphism $f: ({\bf A}, \tw) \rightarrow ({\bf B}, \tw')$ in $\mathcal{NCA}$ we define
$$\m R({\bf A}, \tw)=({\bf A}_{\tw}, \tw(\ln e)) \ \ \mbox{ and } \ \  \m R(f) \mbox{ to be the restriction of } f \mbox{ to }\m A_{\tw}.$$
It can be shown that $\m R$ is a functor from $\mathcal{NCA}$ to $\mathcal{RL}_\text{cy}$.

\medskip

We will show that the functors $\m R$ and $\m T$ form an adjunction. For that consider, for $(\m L, \imath) \in \mathcal{RL}_\text{cy}$ and $({\bf A}, \tw)\in \mathcal{NCA}$, the functions
$$\psi_{(\m L, \imath)}: (\m L, \imath) \rightarrow\m R\m T(\m L, \imath) \text{ given by } \psi_{(\m L, \imath)}(a)= (a, a \ld \imath)$$
and
$$\phi_{({\bf A}, \tw)}: ({\bf A}, \tw)\rightarrow \m T\m R({\bf A}, \tw) \text{ given by } \phi_{({\bf A}, \tw)}(x)= (\tw(x) , \tw(\ln x)).$$

\begin{thm}
Let $(\m L, \imath) \in \mathcal{RL}_\text{cy}$. The function $\psi_{(\m L, \imath)}$ is an isomorphism. Therefore the composition $\m {RT}$ of functors is naturally isomorphic to the identity functor on $\mathcal{RL}_\text{cy}$, via $\psi_{(\m L, \imath)}^{-1}$.
\end{thm}

\begin{proof}
 For $(\m L, \imath) \in \mathcal{RL}_\text{cy}$, we have $\m T(\m L, \imath)= ({\bf Tw}(L,\imath), \tw_\text{Tw})$, where ${\bf Tw}(L,\imath)=\{(a,b) \in L^2: ab \leq \imath\}$. Applying $\m R$ to that we obtain $({\bf Tw}(L,\imath)_{\tw_\text{Tw}}, \tw_\text{Tw}(\ln (e, \imath)))$, where ${\bf Tw}(L,\imath)_{\tw_\text{Tw}}=\{(a, a \ld \imath): a \in L\}$ and also $\tw_\text{Tw}(\ln (e, \imath))=\tw_\text{Tw}(\imath, e)=(\imath, \imath \ld \imath)$.
 We will write $\psi$ for  $\psi_{(\m L, \imath)}$.  Clearly, $\psi: (\m L, \imath) \rightarrow ({\bf Tw}(L,\imath)_{\tw_\text{Tw}}, (\imath, \imath \ld \imath))$ is a bijection. Also,
\begin{itemize}
    \item $\psi(a) \jn \psi(b) = (a, a \ld \imath) \jn (b, b \ld \imath) = (a \jn b, a \ld \imath \wedge b \ld \imath) = (a\jn b, (a \jn b) \ld \imath) =\psi(a \jn b)$
    \item $\psi(a)\psi(b)  = (a, a \ld \imath)(b, b \ld \imath)= (a\cdot b,  (b\ld\imath)\rd a)\wedge (b\ld(a\ld\imath))= (a\cdot b,\imath\rd(a\cdot b) \wedge (a\cdot b)\ld\imath )=(ab, ab \ld \imath)=\psi(ab)$
    \item $\psi(a) \wedge_{\tw_\text{Tw}} \psi(b)  = (a, a \ld \imath) \wedge_{\tw_\text{Tw}} (b, b \ld \imath)= \tw_\text{Tw}(a \wedge b, a \ld \imath \jn b \ld \imath)= (a\wedge b, (a \wedge b) \ld \imath)=\psi(a \jn b)$
    \item $\psi(a)\ld_{\tw_\text{Tw}}\psi(b)  = \tw_\text{Tw}((a, a \ld \imath)\dld(b, b \ld \imath)) = \tw_\text{Tw}(a \ld b \mt (a \ld \imath) \rd (b \ld \imath),(b \ld \imath)a) =\\ \tw_\text{Tw}(a \ld b,(b \ld \imath)a)=\psi(a\ld b)$
    \item $\psi(b)\rd_{\tw_\text{Tw}}\psi(a)  = \tw_\text{Tw}((b, b \ld \imath)\drd(a, a \ld \imath)) = \tw_\text{Tw}(b \rd a \mt (b \ld \imath) \ld (a \ld \imath),a(b \ld \imath)) = \\ \tw_\text{Tw}(b \rd a,a(b \ld \imath))=\psi(b\rd a)$
\end{itemize}

For both divisions, we used the fact that $ a(a \ld b)  (b \ld \imath) \leq \imath$, that $(\imath \rd b) (b \rd a) a \leq \imath$ and cyclicity, to show that $a \ld b\leq (a \ld \imath) \rd (b \ld \imath)$ and $b \rd a \leq (b \ld \imath) \ld (a \ld \imath)$. Furthermore, $\psi(e)=(e,e\ld \imath)=(e, \imath)$ and  $\psi(\imath)= (\imath, \imath \ld \imath)$.

Therefore, $\psi$ is an isomorphism in $\mathcal{RL}_\text{cy}$.
\end{proof}

We will show that every algebra $({\bf A}, \tw)\in \mathcal{NCA}$ can be embedded in $\m {TR}({\bf A}, \tw)$. This means that it can be represented by a twist-product, in the following sense.

\begin{definition}\label{def:twist-product} A  \textsl{twist-product over} $(\m L, \imath) \in \mathcal{RL}_\text{cy}$  is a subalgebra of ${\bf {Tw}(L,\imath)}$ that contains $\tw_\text{Tw}[\text{Tw}(L,\imath)]$, i.e. all elements of the form $(a,\imath \rd a \mt a \ld \imath)$ for $a\in L$.\end{definition}

From Theorem \ref{twists} we have the algebra of pairs
$$    {\bf Tw(\m A_\tw,\imath)} = \left\{\left(\tw(x), \tw(y)\right)\in A_\tw \times A_\tw^\partial: \tw(x)\cdot \tw(y)\leq \imath\right\}$$
where $\imath= \tw(\ln e)$. 

\begin{thm}\label{Teorema_representacion}
Let  $({\bf A},\tw)\in \mathcal{NCA}$ and let $\imath=\tw(\sim e)$. The function $\phi_{({\bf A}, \tw)}: A \to Tw(\m A_\tw, \imath)$ given by the prescription $$x\mapsto (\tw(x),\tw(\sim x))$$ is an injective homomorphism from $(\m A, \tw)$ to $(\m {Tw}(\m A_\tw, \imath), \tw_{Tw})$. In particular, $({\bf A},\tw)$ is isomorphic to a twist-product over $(\m A_\tw,\imath)$.
\end{thm}
\begin{proof} We will write $\phi$ for $\phi_{({\bf A}, \tw)}$. To prove that the mapping is well-defined, observe that $x \cdot \ln x \leq x \cdot (x\ld(\ln e))\leq \ln e$ and therefore
\begin{align*}  \tw(x) \cdot \tw(\ln x) = \tw(x \cdot \ln x) \leq \tw(\ln e) = \imath.\end{align*}

Recall that the neutral element on $\text{Tw}(\m A_\tw,\imath)$ is the pair $(e,\imath).$ Clearly $\phi (e)=(\tw(e), \tw(\ln e))=(e, \imath).$
Using the fact that De Morgan laws hold in ${\bf A}$ one can easily check that $\wedge$ and $\vee$ are preserved by the mapping $\phi.$ Also the preservation of $\ln$ is straightforward from the definition of the mapping and the fact that $\ln$ is an involution in ${\bf A}$. Due to Lemma \ref{lem:propiedades_involucion} we only need to check that $\cdot$ is preserved.
Observe that for all $x, y\in A, $ $$\phi(x\cdot y)=(\tw(x\cdot y), \tw(\ln(x\cdot y)))$$ and
$$\phi(x)\cdot \phi(y)= (\tw(x) \cdot \tw(y), \tw(\ln y)\rd_\tw \tw(x)\wedge_\tw \tw(y)\ld_\tw \tw(\ln x)). $$
As $\tw(x\cdot y)=\tw(x) \cdot \tw(y)$, we just need to prove that the second coordinates of both pairs coincide.
An application of Equation \eqref{eq:nelson'}, Lemma \ref{lem:propoconucleous}  and Lemma \ref{lem:propiedades_involucion} give:
\begin{align*}
   \tw(\ln y)\rd_\tw \tw(x) \wedge_\tw \tw(y)\ld_\tw \tw(\ln x)
&= \tw( \tw(\tw(\ln y)\rd \tw(x))\wedge \tw(\tw(y)\ld \tw(\ln x)) )\\
&\hspace{-1cm}= \tw( \tw(\ln y \rd \tw(x)) \wedge \tw(\tw(y)\ld \ln x) )\\
&\hspace{-1cm}= \tw( (\ln y \rd \tw(x))\wedge (\tw(y)\ld \ln x) )\\
&\hspace{-1cm}= \tw(((\ln y)\drd x)\wedge (y\dld (\ln x)))\\
&\hspace{-1cm}= \tw(\ln y\rd x) = \tw(\ln (x\cdot y)). \end{align*}

\smallskip

To see that the morphism $\phi$ is injective, observe that if $x,y\in A$ are such that $\phi(x)=\phi(y)$, then $\tw(x)=\tw(y)$ and $\tw(\ln x)=\tw(\ln y)$ and Lemma \ref{quasiequation} implies injectivity.

\smallskip

For the last part, observe that
$\tw(\ln \tw(x))=\tw(\tw(x)\ld\ln e)
=\tw(\tw(x)\ld \tw(\ln e))
=\tw(x)\ld_\tw \imath,$
so using the cyclicity of $\imath=\tw(\ln e)$, we get
$\phi(\tw(x))
= (\tw(\tw(x)), \tw(\ln \tw(x)))
=(\tw(x),\tw(x)\ld_\tw \imath)=\tw_\text{Tw}(\phi(x))$.
\end{proof}

The function $\phi_{({\bf A}, \tw)}$ is not always an isomorphism, so the functors $\m R$ and $\m T$ do not form an equivalence between the categories $\mathcal{RL}_\text{cy}$ and $\mathcal{NCA}$. For example,  consider the set $S= Tw({\bf G}_3,0)\setminus \{(0,0)\}$  which is the universe of a subalgebra ${\bf S}$ of ${\bf Tw}({\bf G}_3,0)$, such that $\tw_{Tw}(S)=G_3$; see Figure \ref{3_element_twists}. Then the function $$\phi_{({\bf S},\tw_\text{Tw})}: ({\bf S},\tw_\text{Tw}) \rightarrow \m T \m R({\bf S}, \tw_\text{Tw})$$ is an embedding from ${\bf S}$ into the twist-product over $({\bf G}_3,0)$ that is not an isomorphism.
However, $\m R$ and $\m T$  form an adjunction.

 \begin{thm}\label{adjunction}
 The functor $\m R$ is left adjoint of the functor $\m T$. They form an adjunction between the categories $\mathcal{RL}_\text{cy}$ to $\mathcal{NCA}$, with unit $\phi$ and counit $\psi^{-1}$.
 \end{thm}

 \begin{proof}
 We need to show that for every $({\bf A}, \tw)\in \mathcal{NCA}$ and $(\m L, \imath) \in \mathcal{RL}_\text{cy}$ we have:
 $$1_{\m R(\m A, \tw)}=\psi^{-1}_{\m R(\m A, \tw)} \circ \m R (\phi_{(\m A, \tw)}) \ \mbox{ and } \ \ 1_{\m T(\m L, \imath)}=\m T(\psi^{-1}_{(\m L, \imath)}) \circ \phi_{\m T(\m L, \imath)} $$
For the second identity, recall that $\m T(\m L, \imath)=(\m {Tw}(\m L, \imath), \tw_{Tw})$.
For all $a,b \in  {Tw}(\m L, \imath)$, we have
\begin{align*}\m T(\psi^{-1}_{(\m L, \imath)}) \circ \phi_{\m T(\m L, \imath)}(a,b) &=
\m T(\psi^{-1}_{(\m L, \imath)}) ( \phi_{(\m {Tw}(\m L, \imath), \tw_{Tw})} (a,b) )\\
& = \m T(\psi^{-1}_{(\m L, \imath)}) (\tw_{Tw}(a,b),\tw_{Tw}(\ln(a,b)))\\
& = \m T(\psi^{-1}_{(\m L, \imath)}) (\tw_{Tw}(a,b),\tw_{Tw}(b,a))\\
& = \m T(\psi^{-1}_{(\m L, \imath)}) ((a,a \ld \imath),(b,b\ld \imath))\\
& = (\psi^{-1}_{(\m L, \imath)}(a,a \ld \imath),\psi^{-1}_{(\m L, \imath)}(b,b\ld \imath))\\
& =(a,b)
\end{align*}
  The first one can be written as
 $1_{(\m A_\tw, \tw(\ln e))}=\psi^{-1}_{(\m A_\tw, \tw(\ln e))} \circ \phi{(\m A, \tw)}|_{\m A_\tw} .$
 For all $x \in \m A_\tw$,
 \begin{align*}
 \psi^{-1}_{(\m A_\tw, \tw(\ln e))} \circ \phi{(\m A, \tw)}|_{\m A_\tw}(x) & =
 \psi^{-1}_{(\m A_\tw, \tw(\ln e))} (\phi{(\m A, \tw)}|_{\m A_\tw}(x) )\\ & =
  \psi^{-1}_{(\m A_\tw, \tw(\ln e))} (\tw(x), \tw(\ln x) )\\ & =
    \psi^{-1}_{(\m A_\tw, \tw(\ln e))} (x, x \ld \tw(\ln e) )\\ & =
    x
 \end{align*}
 We used the fact that $x=\tw(x)$ and that $\tw(\ln \tw(x))=\tw(\tw(x)\ld \tw(\ln e))=\tw(x)\ld_\tw  \tw(\ln e))$, which was already mentioned in the proof of Theorem~\ref{Teorema_representacion}.
 \end{proof}

\subsection{Rasiowa-style presentation}

The class of algebras we study is motivated by Nelson lattices and
paraconsistent Nelson lattices. The original representation of these
algebras in terms of twist-products is carried out by considering
the algebra ${\bf A}$ and defining an equivalence relation on ${\bf
A}$ which turns out to be a congruence with respect to some of the
original operations on ${\bf A}$  (see
\cite{Fid,Sen90,Odi-book,BusCig00}). Although our presentation has a
different flavor, it can be compared to the original presentations
of Nelson lattices (as in \cite{Rasi}, \cite{Rasibook}) and of
paraconsistent Nelson lattices (as in \cite{Odi-book}). We connect
these ideas in this section.

\medskip

Let $({\bf A}, \tw)\in \mathcal{NCA}$. Note that the implications $\ld$ and $\rd$ are not preserved by the Nelson conucleus $\tw.$ However, we will prove that the operations
$$x \dld  y= \tw(x) \ld y \ \mbox{ and  } \ \  y \drd x =  y\rd \tw(x)$$
are mapped into the quotient properly.

\begin{lem}\label{lem:imp_respects_congruence} For a residuated lattice ${\bf A}$ and a Nelson conucleus $\tw$ on ${\bf A}$ we have that the map $$n: A \rightarrow A_{\tw}, \ \ \mbox{ where  } \ \ n(x)=\tw(x),$$ is a homomorphism from  the algebra
    ${\bf \bar{A}}= \left(A, \mt, \jn ,\cdot, \dld, \drd,  e\right)$  to the algebra ${\bf A}_{\tw}=(\tw[{\bf A}], \mt_\tw, \jn , \cdot, \ld_{\tw}, \rd_{\tw},  \tw(e)).$

\end{lem}
\begin{proof}   All conuclei produce homomorphisms for meet and $e$, as explained in the first section, while  (T\ref{Tsup}) and (T\ref{Tprod}) give the homomorphism property for join and multiplication.
From  equation (C\ref{Cimpl}) we get $$\tw(x \dld y)=\tw(\tw(x) \ld y)=\tw(\tw(x) \ld \tw(y))=\tw(x) \ld_{\tw} \tw(y).$$ Analogously we prove $\drd$.
\end{proof}

With this result in mind, we consider the relations $\theta$  and $\preceq$ on $A$ defined by:
\begin{align}\label{eq:theta}
x \theta y \ &\mbox{ iff } \  \tw(x)=\tw(y) \\ x\preceq y \ &\mbox{ iff } \ \tw(x)\leq \tw(y) \ (\mbox{iff } \ \tw(x)\leq y).
\end{align}

\begin{lem} \label{preorder_bis}
    If $\tw$ is a Nelson conucleus on $\m A$, then the following hold.
    \begin{enumerate}

        \item The relation $\preceq$ is a preorder compatible with the operations of $\bf \bar{A}$.
        \item The relation $\theta$ is a congruence on ${\bf \bar{A}}$, which is the kernel of the map $n$. Thus, ${\bf \bar{A}} / \theta$ is isomorphic to ${\bf A}_{\tw}$.
    \end{enumerate}
\end{lem}
\begin{proof}
    The fact that $\preceq$ is a preorder is trivial. For the compatibility of $\preceq$ with multiplication, we have $x \preceq y$ and $z \preceq w$ implies $ \tw(x)\leq \tw(y)$ and $ \tw(z)\leq \tw(w)$, so $ \tw(x)\tw(z)\leq \tw(y)\tw(w)$, hence $ \tw(xz)\leq \tw(yw)$, by  (T\ref{Tprod}). Consequently we obtain $xz \preceq yw$. The compatibility with  $\vee$ and $\wedge$ is similar.
    To prove that $\preceq$ is compatible with  $\dld$ and $\drd$ we observe that if $\tw(x) \leq \tw(x')$ and $\tw(y) \leq \tw(y')$, then $(x'\dld y)\preceq (x\dld y')$ and $(y\drd x)\preceq (y'\drd x)$.
    We  give the proof for the case of $\dld$, as the other case is analogous. If $\tw(x)\leq \tw(x')$ and $\tw(y)\leq \tw(y')$, by an application of (C\ref{Cimpl}) in Lemma \ref{lem:propoconucleous} and the definition of $\dld$ we obtain
    $$\tw(x'\dld y)=\tw(\tw(x')\ld y)=\tw(\tw(x')\ld \tw(y))\leq \tw(\tw(x)\ld \tw(y'))=\tw(\tw(x)\ld y')=\tw(x\dld y').$$
        The claim about $\theta$ follow directly from the previous item and  from Lemma \ref{lem:imp_respects_congruence}.
\end{proof}

\begin{lem}\label{preordimp} For a residuated lattice ${\bf A}$ and a Nelson conucleus $\tw$ on ${\bf A}$  we have $$x\preceq y\ \mbox{ iff }\ \ (x\dld y)\dld(x\dld y)\leq (x\dld y)\ \mbox{ iff }\ \ (y\drd x)\drd(y\drd x)\leq (y\drd x).$$
\end{lem}
\begin{proof} If $x\preceq y$ then $\tw(x)\leq y$, so $e\leq \tw(x)\ld y$ and by Equation (C\ref{Cneutral}),
    \begin{align*}(x\dld y)\dld(x\dld y) &= \tw(\tw(x)\ld y)\ld(\tw(x)\ld y)\\
    &\leq \tw(e)\ld(\tw(x)\ld y) = (\tw(x)\cdot\tw(e))\ld y = \tw(x)\ld y\\
    &= x\dld y.\end{align*}
    Now assume $(x\dld y)\dld(x\dld y)\leq (x\dld y)$. Since $\tw(x\dld y) \leq  (x\dld y)$, we have
    \begin{align*} e \leq \tw(x\dld y) \ld  (x\dld y) = (x\dld y)\dld(x\dld y)\end{align*}
    and so $e\leq (x\dld y)=\tw(x)\ld y$. Therefore we obtain $\tw(x)\leq y$ and  $x\preceq y$. The equivalence for $\drd$ is analogous.
\end{proof}

We propose the following definition, which is phrased in terms of $\dld$, $\drd$, $\preceq$ and $\theta$, as a natural generalization of the definitions of  Nelson lattices and paraconsistent Nelson lattices.

\medskip

A \textbf{ Rasiowa-type algebra}  ${\bf \bar{A}}=(A, \vee, \wedge, \cdot, \dld, \drd, \sim, e)$ is an algebra with five binary operations $\vee, \wedge, \cdot, \dld, \drd$, a unary operation $\sim$ and a constant $e$, that satisfies:
\begin{enumerate}
    \item[(R1)] $(A, \vee, \wedge)$ is a lattice and $\ln$ is a De Morgan involution on it, namely $\ln \ln x = x$ and $\ln (x \vee y) = \ln x \wedge \ln y$;
    \item[(R2)] the relation $\preceq$ is a preorder, where  $x\preceq y$ if and only if $(x\dld y)\dld(x\dld y)\leq (x\dld y)$, and also if and only if $(y\drd x)\drd(y\drd x)\leq (y\drd x)$;
    \item[(R3)] the equivalence relation $\theta$ induced by $\preceq$ is a congruence on the algebra ${\bf \bar{A}}=(A,\vee,\wedge,\cdot,\dld, \drd, e)$ and the quotient  ${\bf \bar{A}}/\theta$ is a residuated lattice;
    \item[(R4)] $\ln (x \dld y) \mathrel{\theta} (\ln y \cdot x)\ $, $\ \ln (y \drd x) \mathrel{\theta} (x\cdot\ln y)\ $ and $\ \ln(x \cdot y) \mathrel{\theta} (y \dld\ln x)\wedge  (\ln y\drd x)$;
    \item[(R5)] $x \leq y$ if and only if $x \preceq y$ and $\ln y \preceq \ln x$.
    \item[(R6)]  for each $x\in A$,  $(x\dld\ln e) = (\ln e\drd x).$
\end{enumerate}

Lemmas \ref{lem:imp_respects_congruence}, \ref{preorder_bis}, \ref{preordimp}  and \ref{lem:divisions} yield the next immediate result:

\begin{thm}\label{NelsonStructure} If $({\bf A}, \tw)\in \mathcal{NCA}$, then for $x \dld  y= \tw(x) \ld y$ and$y \drd x =  y\rd \tw(x)$, the structure 
$$\overline{{\bf A}}=(A, \vee, \wedge, \cdot, \dld, \drd, \sim, e)$$
 is a Rasiowa-type algebra.
\end{thm}

To close this circle of ideas we prove the following theorem. We use $[x]$ for the equivalence class of $x$ with respect to $\theta$.

\begin{thm}
    Assume ${\bf \bar{A}}=(A, \vee, \wedge, \cdot, \dld, \drd, \sim, e)$ is a Rasiowa-type algebra. After setting
    \begin{align*}x\ld y=\ln(\ln y\cdot x )  \mbox{ and } \ y\rd x=\ln (x\cdot \ln y),\end{align*}
    the algebra ${\bf A}=(A, \vee, \wedge, \cdot, \ld, \rd, \sim, e)$ is an involutive residuated lattice, which is isomorphic to a subalgebra of ${\bf Tw}({\bf \bar{A}}/\theta, [\sim e])$.   Moreover, if we set
    \begin{align*}\tw(x)=\ln(x\dld \ln e)\end{align*} 
    we have that the pair $({\bf A},\tw)$ is in $\mathcal{NCA}$. 
\end{thm}

\begin{proof}
    We define a map $h$ on $A$ by
    \begin{align*}h(x)=([x], [\ln x])\end{align*}
and prove that it is an injective homomorphism from $\bf A$ into ${\bf Tw}({\bf \bar{A}}/\theta,[\sim e])$, and therefore show that $\bf A$ is an involutive residuated lattice.

    \medskip

    First we observe that the image of $h$ is in the set ${ Tw}({ \bar{A}}/\theta,[\sim e])$. Indeed, we need to see that  for each $x\in A$ we have
    \begin{equation}\label{eq:in_the_algebra}
    [x] \cdot [\ln x] \vee [\ln x] \cdot [x] \le [\ln e].
    \end{equation}
 To this aim, by (R4) we get that $[\ln x]=[\ln (e\cdot x)]\le  [x\dld \ln e]$ and therefore, since ${\bf \bar{A}}/\theta$ is a residuated lattice by (R3), we get
    \begin{align*}[x]\cdot [\ln x] \le [x]\cdot [x\dld \ln e] = [x]\cdot  ([x]\dld [\ln e]) \le  [\ln e], \end{align*}
    and similarly $[\ln x] \cdot [x] \leq [\ln e]$, ensuring that the inequality (\ref{eq:in_the_algebra}) holds.

    The injectivity of $h$ follows directly from (R5). 

    \medskip

    Note that (R1) implies that $h$ preserves $\vee, \wedge$ and $\ln.$ For the product, because of (R4) observe that for any $x,y\in A$ we have
\begin{align*}h(x)\cdot h(y) & =([x], [\ln x])\cdot ([y], [\ln y]) =([x\cdot y],  [(y \dld\ln x)\wedge  (\ln y\drd x)])\\
&= ([x\cdot y], [\ln(x\cdot y)]) = h(x\cdot y).
\end{align*}

We also have $h(e)=([e],[\ln e])$, which is the identity element of ${\bf Tw}({\bf \bar{A}}/\theta,[\sim e])$. As $\ld$ and $\rd$ are defined in terms of the product and the involution, we conclude that $h$ is an injective morphism from $\bf A$ into ${\bf Tw}({\bf \bar{A}}/\theta,[\sim e])$, and in particular $\bf A$ is an involutive residuated lattice.

\medskip

Observe now that as a consequence of (R1) and (R4) we have that
$$x\ \theta \ x\cdot e= x\cdot \ln\ln e\ \theta\  \ln( \ln e\drd x) \ \mbox{ and } \ \  x\ \theta \  e\cdot x=\ln\ln e\cdot x\ \theta\   \ln(x\dld \ln e),$$
so $ x\ \theta \ln( \ln e\drd x)\ \theta  \ \ln(x\dld \ln e).$
Therefore
\begin{align*}h(\tw(x)) &= h(\ln(x\dld \ln e) \vee \ln(\ln e \drd x))\\
                                                &= \left([\ln(x\dld \ln e) \vee \ln(\ln e \drd x)],[(x\dld \ln e) \wedge (\ln e \drd x)]\right)\\
                                                &= \left([x],([x]\dld [\ln e]) \wedge ([\ln e] \drd [x])\right)\\
                                                &= \tw_{Tw}(h(x)).\end{align*}
By (R6) $x\dld \ln e=\ln e \drd x$, thus $[\ln e]$ is cyclic,  $\tw_{Tw}$ is a Nelson conucleus and so is $\tw$.

\end{proof}

Therefore, Rasiowa-type algebras, which include Nelson lattices as defined by Rasiowa and eN4-lattices (see \cite{Sen90,BusCig00}), are term equivalent to Nelson conucleus algebras. Thus the Rasiowa-style definition can be exchanged by one that has the advantage of being internal to ${\bf A}$ and can be seen as  providing representatives for Rasiowa's equivalence classes $[x]$ via the elements $\tw(x) \in [x]$. We also note that in the above proof we opted for the easier argument of embedding the algebra into a twist product, but the term equivalence can be proved directly; for example, from (R5) in order to prove associativity of multiplication it suffices to show that $[(xy)z]=[x(yz)]$ and $[\ln((xy)z)]=[\ln(x(yz))]$.

\section{Algebras with a term definable Nelson conucleus}\label{Sec4}

 Nelson residuated lattices, paraconsistent Nelson residuated lattices and Kalman lattices, are particular examples of algebras in which the Nelson conucleus can be defined by a term, i.e., there is a definable term function $\tw$ such that for each ${\bf A}$ in the corresponding class the pair $({\bf A}, \tw)$ is in $\mathcal{NCA}$. Before analyzing that, we note that the theory we have developed is broader and admits algebras where the conucleus is not definable by a term.

\begin{ex}
    Consider the twist structure ${\bf \text{Tw}}(\text{\L}_3,0)$ shown in Figure \ref{23_element_twists} and in Figure \ref{notaterm} below. Observe that $\tw_{\text{Tw}}(a,0)=(a,a)$ so $({\bf \text{Tw}}(\text{\L}_3,0),\tw_{\text{Tw}})$ is in $\mathcal{NCA}.$  However $\text{Tw}(\text{\L}_3,0)\setminus\{(a,a)\}$ is the universe of a subalgebra of ${\bf \text{Tw}}(\text{\L}_3,0)$ (as a $\tw_{\text{Tw}}$-less reduct), and therefore $\tw_{\text{Tw}}$ cannot be a term function, since the element $\tw_{\text{Tw}}(a,0)$ cannot be obtained from operations applied to $(a,0)$ and the identity $(1,0)$.

    To show that it is a subalgebra, observe that it is a chain, therefore closed under $\wedge$ and $\vee$, and also it is closed under $\ln$. As the implication can be defined in terms of the product and the involution, it only remains to show that it is closed under the product, and we do this in Figure \ref{notaterm}.

    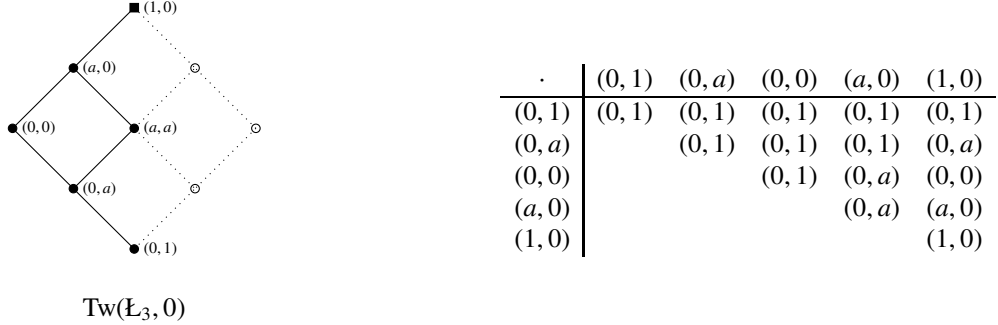
\begin{figure}[ht]
        \centering
        \begin{minipage}{0.4\textwidth}
            \begin{tikzpicture}[scale=0.8]
            \draw[black, fill=black] (0-6,-6+0) circle (0.07 cm);
            \draw[black, fill=black] (-2-6,-6+0) circle (0.07 cm);
            \draw[black] (2-6,-6+0) circle (0.07 cm);
            \draw[black, fill=black] (-1-6,-6+1) circle (0.07 cm);
            \draw[black] (1-6,-6+1) circle (0.07 cm);
            \draw[black, fill=black] (-1-6,-6+-1) circle (0.07 cm);
            \draw[black] (1-6,-6+-1) circle (0.07 cm);
            \draw[black, fill=black] (0-6,-6+2) +(-2pt,-2pt) rectangle +(2pt,2pt) ;
            \draw[black, fill=black] (0-6,-6+-2) circle (0.07 cm);
            \draw[black] (0-6,-6+0)--(-1-6,-6+1);
            \draw[dotted] (0-6,-6+0)--(1-6,-6+1);
            \draw[black] (0-6,-6+0)--(-1-6,-6+-1);
            \draw[dotted] (0-6,-6+0)--(1-6,-6+-1);
            \draw[black] (-2-6,-6+0)--(-1-6,-6+-1);
            \draw[black] (-2-6,-6+0)--(-1-6,-6+1);
            \draw[dotted] (2-6,-6+0)--(1-6,-6+-1);
            \draw[dotted] (2-6,-6+0)--(1-6,-6+1);
            \draw[black] (0-6,-6+-2)--(-1-6,-6+-1);
            \draw[dotted] (0-6,-6+-2)--(1-6,-6+-1);
            \draw[black] (0-6,-6+2)--(-1-6,-6+1);
            \draw[dotted] (0-6,-6+2)--(1-6,-6+1);
            \node[right] at (0-6,-6+0) {\tiny$(a,a)$};
            \node[right] at (-2-6,-6+0) {\tiny$(0,0)$};
            \node[right] at (-1-6,-6+1) {\tiny$(a,0)$};
            \node[right] at (-1-6,-6+-1) {\tiny$(0,a)$};
            \node[right] at (0-6,-6+2) {{\tiny$(1,0)$}};
            \node[right] at (0-6,-6+-2) {\tiny$(0,1)$};
            \node at (0-6,-6+-3) {$\text{Tw}(\text{\L}_3,0)$};
            \end{tikzpicture}
        \end{minipage}
        \begin{tabular}{c|ccccc}
            $\cdot$ &$(0,1)$& $(0,a)$ & $(0,0)$ & $(a,0)$ & $(1,0)$  \\
            \hline
            $(0,1)$ & $(0,1)$ & $(0,1)$ & $(0,1)$ & $(0,1)$ & $(0,1)$ \\
            $(0,a)$ &  & $(0,1)$ & $(0,1)$ & $(0,1)$ & $(0,a)$ \\
            $(0,0)$ &  &  & $(0,1)$ & $(0,a)$ & $(0,0)$ \\
            $(a,0)$ &  &  &  & $(0,a)$ & $(a,0)$ \\
            $(1,0)$ &  &  &  &  & $(1,0)$ \\
        \end{tabular}
        \caption{An example where the conucleus $\tw$ cannot be a term function.}
        \label{notaterm}
    \end{figure}
\end{ex}

\subsection{The defining terms and Kalman residuated lattices}

Note that all the motivating examples are of commutative involutive residuated lattices, and the residuated lattices ${\bf A}_\tw$ are integral. They also share some similarities with respect to the terms defining the conucleus. Two particular cases follow from the next results, the first one in Lemma \ref{Kalmanconucleus} and the second in Theorem \ref{Nelsoncases}.

\begin{lem}\label{lem:integralcase}\label{lem:3potency}
Let $({\bf A},\tw)\in \mathcal{NCA}$. Then  $(x\wedge e)^2\leq \tw(x)$ and if  $\tw(x)\leq e$ for all $x$, then
\begin{enumerate}
    \item $(x\wedge e)^2\leq \tw(x)\leq x\wedge e$.
    \item If $e \leq \ln e$, then $e=\ln e$ and $\tw(x)=x\wedge e$.
    \item If $\tw(x)^2=\tw(x)$, then $\tw(x)=(x\wedge e)^2$.
\end{enumerate}

\end{lem}

\begin{proof}
Using (T\ref{Tmulteq}), we have
\begin{align*}(x\wedge e)^2 = \tw(x\wedge e)\cdot (x\wedge e) \vee (x\wedge e)\cdot\tw(x\wedge e) \leq \tw(x)e\vee e\tw(x)=\tw(x).\end{align*}
If $\tw(x)\leq e$ then $\tw(x)\leq x\mt e$, so
$(x\wedge e)^2  \leq \tw(x) \leq x \mt e.$ With this in mind, we only need to show one inequality for (2) and for (3).

Observe first that if $\tw(x)\leq e$ holds, then $\ln e\ld e = \tw(\ln e)\ld e \wedge e\rd\tw(\ln e)\geq e$, so $\ln e\leq e$.
If $e \leq \ln e$, then by the previous observation $e=\ln e$ and as $\tw(\ln\tw(x))\leq e$, by (T\ref{TNelson}) we get
\begin{align*} (x\wedge e)\ld\tw(x) &=\tw(x\wedge e)\ld\tw(x)\wedge\ln(x\wedge e)\rd \tw(\ln\tw(x))\\ &\geq  \tw(x)\ld\tw(x)\wedge(\ln x\vee \ln e)\rd e \geq e\wedge \ln e = e, \end{align*}
so $x\wedge e\leq \tw(x)$.
If $\tw(x)^2=\tw(x)$, then clearly $\tw(x)=\tw(x)^2\leq (x\wedge e)^2$.

\end{proof}

Observe that if $\tw(x)^2=\tw(x)$ and ${\bf A}$ is commutative (even without assuming $\tw(x)\leq e$), then  using equation (T\ref{Tmulteq}) we have that $x^2=\tw(x)x$, and so
\begin{align*}x^3=x^2\cdot x=\tw(x)xx= \tw(x)x^2= \tw(x)\tw(x)x = \tw(x)x=x^2.\end{align*} This means that if $\tw(x)^2=\tw(x)$ and ${\bf A}$ is commutative, then ${\bf A}$  is $3$-potent (it satisfies $x^3=x^2$).

\begin{lem}\label{Kalmanconucleus} Given a Kalman residuated lattice ${\bf A}$, the term function $\tw(x)=x\wedge e$ defines a Nelson conucleus on ${\bf A}$, and therefore $({\bf A}, \tw)\in \mathcal{NCA}$. Conversely, if $({\bf A}, \tw)\in \mathcal{NCA}$ satisfies commutativity, $\tw(x)\leq e$  and $e\leq \ln e$, then $\bf A$ is a Kalman residuated lattice.
\end{lem}

\begin{proof}
    First assume that ${\bf A}$ is a Kalman residuated lattice and we set  $\tw(x)=x\wedge e$. It is immediate that equations (C\ref{Cord1}), (C\ref{Cidem}) and (C\ref{Cmono}) hold in ${\bf A}$. (C\ref{Cprod}) will follow from (T\ref{Tprod}) and (C\ref{Cneutral}) comes from the fact that $\tw(e)=e\wedge e=e.$ Finally, (T\ref{Tsup}) is (K\ref{Kquasidist}), (T\ref{Tprod}) is (K\ref{Kdist_ast_inf}) and (T\ref{Tmulteq}) follows from commutativity, (K\ref{Kquasi_nelson}) and Lemma \ref{lem:divisions}.

        Now, if $({\bf A}, \tw)\in \mathcal{NCA}$ satisfies commutativity, $\tw(x)\leq e$ and $ e\le \ln e$. Then  by Lemma \ref{lem:integralcase} we have $e=\ln e$ and $\tw(x)=x\wedge e$, and the only equation left to verify that $\bf A$ is a Kalman lattice is the fact that $x \land (y \lor e) = (x \land y) \lor (x \land e)$. As $\tw(x)=x\wedge e$ is a Nelson conucleus,
        \begin{align*}(x \land (y \lor e))\land e &= x \land ((y\lor e)\land e)) = x\land e\\
            ((x \land y) \lor (x \land e))\land e   &= (x\land y \land e)\lor(x\land e) = x\land e\end{align*}
        and
        \begin{align*}\ln(x \land (y \lor e))\land e &= (\ln x\land e) \lor ((\ln y\land e)\land e) = (\ln x\lor \ln y)\land e\\
            \ln((x \land y) \lor (x \land e))\land e    &= ((\ln x\lor \ln y) \land e)\land((\ln x\lor e)\land e) = (\ln x\lor \ln y) \land e\end{align*}
    so by Lemma \ref{quasiequation} we have the equality.
\end{proof}

\subsection{Nelson-type algebras}

In this section we will construct a variety of algebras $\bf A$ with a definable term function $\tw$ such that $({\bf A}, \tw)$ is in $\mathcal{NCA}$. This variety will encompass both Nelson residuated lattices and Nelson paraconsistent lattices.
To motivate our definition, we first rewrite what we need from Lemma \ref{lem:integralcase}.

\begin{lem}If $({\bf A}, \tw)\in \mathcal{NCA}$ satisfies $\tw(x)\leq e$ and $\tw(x)^2=\tw(x)$, then $\tw(x)=(x\wedge e)^2$ and ${\bf A}$ is a residuated lattice satisfying:
\begin{enumerate}[{(N}1{)}]
    \item\label{eq:nelsonnt} $xy=(x\wedge e)^2 y \jn x (y\wedge e)^2$,
    \item\label{eq:producto_cong} $(x y\wedge e)^2=(x\wedge e)^2(y\wedge e)^2$.
 \end{enumerate}
Moreover, ${\bf A}_\tw$ is an integral residuated lattice where $\cdot$ and $ \wedge_\tw$ coincide, i.e. it is a Brouwerian algebra.
Finally, ${\bf A}$ is commutative, distributive, and satisfies $((x \jn y)\wedge e)^2=(x\wedge e)^2 \jn (y\wedge e)^2$.
\end{lem}
\begin{proof}The first facts follow from Lemma \ref{lem:integralcase}. Note that because of integrality of ${\bf A}_\tw$,
$$\tw(x)\wedge_\tw \tw(y)=\tw(x\wedge y)=\tw(x\wedge y)^2\leq \tw(x)\tw(y)\leq \tw(x)\wedge_\tw \tw(y),$$
so ${\bf A}_\tw$ is  a Brouwerian algebra.
Therefore, ${\bf A}_\tw$ is commutative and distributive and by Theorem \ref{Teorema_representacion} $\bf A$ will also satisfy commutativity and distributivity, as these properties are inherited by the twist-product.
\end{proof}

With this in mind, we define the variety $\mathcal{NT}$ of \textbf{Nelson-type algebras}, as the variety of commutative, distributive, involutive residuated lattices satisfying equations (N\ref{eq:nelsonnt})-(N\ref{eq:producto_cong}). We show below that $\mathcal{NT}$ is (term equivalent to) a subvariety of $\mathcal{NCA}$. Considering the term  $\tw(x)=(x\wedge e)^2$, these axioms can be written as:
\begin{enumerate}[{(N'}1{)}]
    \item\label{eq:nelsonnt'}
    $xy=\tw(x)y \jn x \tw(y),$
    \item\label{eq:producto_cong'}$
    \tw(xy)=\tw(x)\tw(y).$
\end{enumerate}

As an immediate consequence of the observation below Lemma~\ref{lem:3potency} we get that for each $n\in \mathbb{N}$, $n\ge 2$ the equation $x^n=x^2$ holds in $\mathcal{NT}.$

\begin{lem}If ${\bf A}\in \mathcal{NT}$, then $\tw(x)=(x\wedge e)^2$ is a Nelson conucleus.
\end{lem}
\begin{proof}
    Trivially  $\tw$ satisfies (C\ref{Cord1}) and (C\ref{Cmono}). Equation  (C\ref{Cidem}), that is, $((x\wedge e)^2\wedge e)^2 = (x\wedge e)^2$ is immediate from Lemma \ref{lem:3potency}, as $(x\wedge e)^2\leq e$ and $(x\wedge e)^4=(x\wedge e)^3=(x\wedge e)^2$ from (N\ref{eq:nelsonnt}). (C\ref{Cprod}) will follow from (N\ref{eq:producto_cong}) and  (C\ref{Cneutral}) follows by observing that $(e\wedge e)^2=e$. To prove  (T\ref{Tsup}), we have to show that
    $(x\wedge e)^2\vee(y\wedge e)^2 = ((x\vee y)\wedge e)^2$
    hold. Since $\cdot$ and $\wedge$ distribute over $\vee$, we have that
    \begin{align*}
            ((x\vee y)\wedge e)^2 &= ((x\vee y)\wedge e)^3\\ &= (x\wedge e)^3\vee (x\wedge e)^2(y\wedge e)\vee (x\wedge e)(y\wedge e)^2\vee (y\wedge e)^3\\&= (x\wedge e)^2\vee (y\wedge e)^2.\end{align*}
    (T\ref{Tprod}) is (N\ref{eq:producto_cong}) and (T\ref{Tmulteq}) follows from (N\ref{eq:nelsonnt}).
\end{proof}

We conclude:

\begin{thm}\label{Nelsoncases} If ${\bf A}\in \mathcal{NT}$, then $\tw(x)=(x\wedge e)^2$ is a Nelson conucleus and $({\bf A}, \tw)\in \mathcal{NCA}$. Conversely, if $({\bf A}, \tw)\in \mathcal{NCA}$ satisfies $\tw(x)^2=\tw(x)\leq e$, then ${\bf A}\in \mathcal{NT}$.\end{thm}

As we are in the commutative setting, we set $x\to y = x\ld y = y\rd x$ and $x\Rightarrow y = x\dld y = y\drd x$.
The following result shows that Nelson-type algebras generalize both Nelson residuated lattices and Nelson paraconsistent residuated lattices.

\begin{lem} We have that:
\begin{enumerate}
    \item The variety of Nelson residuated lattices is term equivalent to the subvariety of integral Nelson-type algebras.
        \item The variety of Nelson paraconsistent residuated lattices is term equivalent to the subvariety of odd Nelson-type algebras.
\end{enumerate}
\end{lem}
\begin{proof}
1.
We will show that if  ${\bf A}=(A, \vee, \wedge, \cdot, \to, \bot, e)$ is a Nelson residuated lattice, after setting $\neg x=x\rightarrow \bot$, the algebra  ${\bf A}=(A, \vee, \wedge, \cdot, \to, \neg, e)$ is in $\mathcal{NT}$  and it is integral.  The definition of Nelson residuated lattices imply that  ${\bf A}=(A, \vee, \wedge, \cdot, \to, \neg, e)$ is commutative, distributive, involutive (NRL1), and integral. Integrality trivializes (N\ref{eq:producto_cong}). For (N\ref{eq:nelsonnt}) we observe that from (NRL1) and (NRL2), 
\begin{align*}xy=\neg\neg (xy) = \neg (x\to \neg y) = \neg(x^2\to \neg y)\vee\neg(y^2\to \neg x)= x^2y \vee y^2x.\end{align*}

\smallskip

Assume now that ${\bf A}=(A, \vee, \wedge, \cdot, \to, \ln, e)$ is in $\mathcal{NT}$ and satisfies integrality. If we set $\bot=\ln e,$ then $$\ln e\vee x=\ln (e\wedge \ln x)=\ln\ln x=x$$ and therefore $\bot$ is the lowest element in $A$. Besides $\ln x=x\to \ln e=x\to \bot$ thus $\neg x=\ln x$ and (NRL1) holds in ${\bf A}$. Therefore ${\bf A}=(A, \vee, \wedge, \cdot, \to, \bot, e)$ is an involutive residuated lattice in the sense of \cite{BusCig00}. Equation (NRL2) follows from integrality and (N\ref{eq:nelsonnt}).

 2. Let ${\bf A}=(A, \vee, \wedge, \cdot, \to, e)$ be an NPc-lattice. From (NPc3)  the term function $\tw(x)=(x\wedge e)^2$ becomes $\tw(x)=x\wedge e$ and clearly  ${\bf A}=(A, \vee, \wedge, \cdot, \to, \ln, e)$ is in $\mathcal{NT}.$

Now consider an algebra  ${\bf A}=(A, \vee, \wedge, \cdot, \to, \ln, e)\in \mathcal{NT}$ satisfying  $e=\ln e$. Then $\ln x=x\to \ln e=x\to e$ and ${\bf A}=(A, \vee, \wedge, \cdot, \to, e)$ is a lattice with involution. As (NPc1) is required we need to check that the other three equations hold in ${\bf A}.$

Equations (NPc2) and (NPc4) are equivalent to (N\ref{eq:producto_cong}) and (N\ref{eq:nelsonnt}) under (NPc3). So it only remains to show that equation (NPc3) follows from (N\ref{eq:nelsonnt})-(N\ref{eq:producto_cong}) and (NPc1). Observe that $(x\wedge e)^2\leq x\wedge e$ always holds, for the other inequality, by (N\ref{eq:nelsonnt}) and $\ln e=e$,
                \begin{align*} (x\wedge e)\to(x\wedge e)^2 &= ((x\wedge e)^2\to (x\wedge e)^2)\wedge ((\ln (x\wedge e)^2\wedge e)^2\to \ln (x\wedge e))\\
                                                                                                     &\geq e\wedge (e\to \ln (x\wedge e)) = e\wedge (\ln x\vee e)= e,\end{align*}
            so $x\wedge e\leq (x\wedge e)^2$.
\end{proof}

Observe that there are algebras in $\mathcal{NT}$ which are neither equivalent to Nelson residuated lattices nor to Nelson paraconsistent residuated lattices. For example, this is the case of ${\bf Tw}({\bf G}_3,a)$ in Figure \ref{3_element_twists}. It is immediate to check that (N\ref{eq:nelsonnt}) and (N\ref{eq:producto_cong}) hold in this algebra but $e=(1, a)\neq (a,1)=\sim e$ and $e=(1,a)<(1,0).$

We can  axiomatize the join of Nelson residuated lattices and NPc-lattices  as the variety $\mathcal{NT}_0$ of Nelson-type algebras satisfying the condition
\begin{enumerate}[{(N}1{)}]
    \setcounter{enumi}{2}
    \item\label{eq:oddintegral} $e\leq \ln e\vee (\ln e\to x)$.
 \end{enumerate}

\begin{lem} The variety $\mathcal{NT}_0$ is the variety generated by integral and odd Nelson-type algebras.\end{lem}
\begin{proof} Let ${\bf A}\in \mathcal{NT}_0$ be subdirectly irreducible. By \cite[Theorem 2.9]{HarRafTsi}, in any subdirectly irreducible commutative residuated lattice $e$ is join-irreducible. 
By distributivity in (N4)
$e = (\ln e \wedge e) \vee ((\ln e\to x) \wedge e),$
so either $e\leq \ln e$, in which case by Lemma \ref{lem:integralcase} we have $e=\ln e$, or $\ln e\leq x$ for all $x$, which implies that $\bf A$ is integral.
\end{proof}

\section{ Sendlewski-style representation}

We present two cases in which we can improve the categorical adjunction to a categorical equivalence.

\subsection{The case of Nelson-type algebras}

Theorem \ref{Teorema_representacion} holds for algebras in $\mathcal{NT}$ with the Nelson conucleus given by $\tw(x)=(x\wedge e)^2$. To improve the  result for this class, we follow Sendlewski's ideas for Nelson algebras (\cite{Sen90}) and Odintsov's ideas for the paraconsistent case (\cite{Odi04}).
Recall that if ${\bf A}=(A,\vee, \wedge, \cdot, \to, \ln, e)$ is in $\mathcal{NT}$, then ${\bf H}_{\bf A}=(\tw[A], \vee, \wedge_\tw, \to_\tw, e)$ is a Brouwerian algebra.

If ${\bf H}$ is a Brouwerian algebra and $F\subseteq H$ is a lattice filter, then $F$ is called a \textbf{Boolean filter} if for every $x,y\in H$ the element $x\vee (x\rightarrow y)$ belongs to $F$; such elements are called \textbf{dense}. The reader may note that if ${\bf H}$ is lower bounded, i.e., it is the $\bot$-free reduct of a Heyting algebra and $F$ is a Boolean filter of ${\bf H}$, then the quotient ${\bf H}/F$ is the $\bot$-free reduct of a Boolean algebra.

\begin{lem}\label{NT1}
Let $\bf H$ be a Brouwerian algebra, $\imath\in H$ and $F$ a Boolean filter of $\m H$. Then the set
$$\text{Tw}(H,\imath,F) = \{(a,b)\in H\times H^\partial: a\wedge b\leq \imath,a\vee b\in F\}$$
is the universe of a subalgebra of ${\bf Tw(H,\imath)}$ which is a twist-product over ${\bf H}$.
\end{lem}
\begin{proof}We only need to show the closure of the operations $\ln$, $\wedge$ and $\cdot$, as $\vee$ and $\to$ can be derived from them. Consider $(a,b),(a',b')\in \text{Tw}(H,\imath,F)$. This means that $a\wedge b\leq\imath$, $a\vee b\in F$, $a'\wedge b'\leq\imath$ and $a'\vee b'\in F$.
\begin{itemize}
    \item[$(\ln)$] this is immediate from the definition.
    \item[$(\wedge)$] $(a,b)\wedge(a',b')=(a\wedge a',b\vee b')$. We have to show that $(a\wedge a')\vee(b\vee b')\in F$. Indeed,
        $ (a\wedge a')\vee(b\vee b')   = (a\vee b\vee b')\wedge (a'\vee b\vee b')\geq (a\vee b)\wedge (a'\vee b')\in F.$
    \item[$(\cdot)$] $(a,b)\cdot(a',b') = (a\wedge a',a\to b'\wedge a'\to b)$, then we have to prove that $(a\wedge a')\vee(a\to b'\wedge a'\to b)\in F$. Recalling that $F$ contains all the dense elements,
        \begin{align*}(a\wedge a')\vee(a\to b'\wedge a'\to b) &=\\
                &\hspace{-4cm}=(a\vee (a\to b'))\wedge(a\vee (a'\to b))\wedge(a'\vee (a\to b'))\wedge(a'\vee (a'\to b))\\
                &\hspace{-4cm}\geq (a\vee (a\to b'))\wedge(a\vee b)\wedge(a'\vee b')\wedge(a'\vee (a'\to b))\in F\end{align*}
\end{itemize}
Finally, for each $a\in H$, the pair $(a, a\to \imath)$ is in $\text{Tw}(H, \imath, F)$ since $a\cdot (a\to \imath)\le \imath$ by residuation and $a\vee (a\to \imath)$ is a dense element contained in $F.$ Thus Definition \ref{def:twist-product} asserts that $\text{Tw}(H, \imath, F)$ is a twist-product over ${\bf H}.$
\end{proof}

To obtain an isomorphism for each ${\bf A}\in \mathcal{NT}$, we need to find a Boolean filter $F$ on ${\bf H}_{\bf A}$ such that ${\bf A}\cong {\bf Tw}({\bf H}_{\bf A},\imath,F)$. Recalling that $x\Rightarrow y = x\dld y = y\drd x$ in the commutative case, we will frequently use the fact that the equation
\begin{equation}\label{eq:implicadebil_inf}
\tw(\ln(x\Rightarrow y))=\tw(\tw(x)\wedge(\ln y))
\end{equation}
 holds for every pair of elements $x, y$ in an algebra ${\bf A}\in \mathcal{NT}$ (it follows from the definition of $\Rightarrow$, the results of Lemma \ref{lem:propiedades_involucion} and the fact that $\wedge_\tw$ and $\cdot$ coincide in $\bf H_{\bf A}$).

\begin{lem} If ${\bf A}\in\mathcal{NT}$ and then the subset $F_A=\{\tw(x\vee\ln x): x\in A\}=\{\tw(w) : \ln w \leq w\} = \{\tw(z) : \tw(\ln z) \leq \tw(z)\}$  of $H_A$ is a Boolean filter.\end{lem}

\begin{proof} For all $x \in A$, we have $\ln (x \jn \ln x)=\ln x \mt x \leq x \jn \ln x$, and if $\ln w \leq w$, then $w \jn \ln w = w$. So, $\{\tw(x\vee\ln x): x\in A\}=\{\tw(w) : \ln w \leq w\}$.
We have that if $\ln w \leq w$, then $\tw(\ln w) \leq \tw(w)$, so $\{\tw(w) : \ln w \leq w\} \subseteq \{\tw(z) : \tw(\ln z) \leq \tw(z)\}$. Also, if $z$ satisfies $\tw(\ln z) \leq \tw(z)$, then $\tw(z \jn \ln z)=\tw(z) \jn \tw(\ln z)=\tw(z)$; so, $\{\tw(z) : \tw(\ln z) \leq \tw(z)\} \subseteq \{\tw(x\vee\ln x): x\in A\}$.

We will prove that $F_A$ is a (lattice) filter. Indeed, as $\tw(e)$ is the top element of ${\bf H}_{\bf A}$, we have $\tw(e) = \tw(e\vee\ln e)\in F_A$.

To prove closure under $\wedge$, take $\tw(w), \tw(z)\in F_A$, with
$\ln w \leq w$ and $\ln z \leq z$. For
\begin{align*}  t &= \ln \left(w\wedge \left(w\Rightarrow \ln(w\vee z)\right)\right)\mt \ln \left(z\wedge \left(z\Rightarrow \ln(w\vee z)\right)\right)\\
    &=(\ln w \jn \tw(w)(w \jn z) ) \mt (\ln z \jn \tw(z)(w \jn z) ),\end{align*}
    we have $ \ln t = \left(w\wedge \left(w\Rightarrow \ln(w\vee z)\right)\right)\vee\left(z\wedge \left(z\Rightarrow \ln(w\vee z)\right)\right).$
 Using divisibility and distributivity,
\begin{align*}
    \tw(\ln t)                  &= \left(\tw(w) \wedge_\tw \tw(\ln(w\vee z))\right)\vee\left(\tw(z)\wedge_\tw \tw(\ln(w\vee z))\right)\\
                                        &= \tw(w\vee z) \wedge_\tw \tw(\ln(w\vee z))\\
                                        &= \tw(\ln w) \wedge_\tw \tw(\ln z).
\end{align*}

Recalling Equation (\ref{eq:implicadebil_inf}) and that $p\leq \ln \ln(p\vee q)$ for all $p,q$ we get
\begin{align*}
    \tw(t)  &= \left(\tw(\ln w)\vee \left(\tw(w)\wedge_\tw \tw(\ln\ln(w\vee z))\right)\right)\wedge\left(\tw(\ln z)\vee \left(\tw(z)\wedge_\tw \tw(\ln\ln(w\vee z))\right)\right)\\
                                        &= \tw(w\vee\ln w) \wedge_\tw \tw(z\vee\ln z)\\
                                        &= \tw(w) \wedge_\tw \tw(z),
\end{align*}
so $\tw(\ln t) \leq \tw(t)$, hence $\tw(t) \in F_A$. Since also $\tw(w) \wedge_\tw \tw(z)=\tw(t)$, we get  $\tw(w) \wedge_\tw \tw(z) \in F_A$.

Additionally, if $\tw(w) \geq \tw(y)$, where $\ln w \leq w$, we have $\ln (w \jn y) = \ln w \mt \ln y \leq w \leq w \jn y$, hence $\tw(w \jn y) \in F_A$. Since we also have $\tw(w \jn y)=\tw(w) \jn \tw(y)=\tw(y)$ we get  $\tw(y) \in F_A$.

 To show that $F_A$ contains all dense elements, we need to see for all $x,y\in A$,  that the element $\tw(x) \vee (\tw(x) \to_\tw \tw(y)) = \tw(x\vee (x\Rightarrow y))$ is in $F_A$. This is true since using Equation (\ref{eq:implicadebil_inf}) we obtain
\begin{align*}\tw(\ln (x\vee (x\Rightarrow y))) &= \tw(\ln x\wedge (\tw(x)\cdot \ln y))= \tw(\ln x) \wedge \tw(x) \wedge \tw(\ln y)\\
&\leq \tw(x) \vee (\tw(x) \to_\tw \tw(y))= \tw(x\vee (x\Rightarrow y)).\end{align*}
\end{proof}

\begin{thm}\label{NT2} Let ${\bf A}\in\mathcal{NT}$. If ${\bf H_A}$, $\imath$ and $F_A$ are as before, then ${\bf A}\cong {\bf Tw({\bf H}_{\bf A},\imath,F_A)}$.\end{thm}
\begin{proof}From Theorem \ref{Teorema_representacion}, it is sufficient to show that $\phi_{\bf A}(A)=\text{Tw}(A_\tw,\imath,F_A)=\text{Tw}(H_A,\imath,F_A)$. One inclusion is immediate, for if $x\in A$, then $\tw(x\vee\ln x) \in F_A$ by definition of $F_A$.

For the other, consider $(\tw(x),\tw(y))\in \text{Tw}(H_A,\imath,F_A)$. We have to find $z\in A$ such that $\tw(z) = \tw(x)$ and $\tw(\ln z) = \tw(y)$. As $\tw(x)\vee \tw(y)\in F_A$ by hypothesis, let $w\in A$ be such that $\tw(x)\vee \tw(y) = \tw(w\vee\ln w)$. Define now
\begin{align*}
    z &= [(w\wedge \ln w) \vee \ln\left(x\Rightarrow \tw(y)\right) \vee \ln\left(y\Rightarrow \tw(x)\right)]\wedge \left(y\Rightarrow \tw(x)\right).
\end{align*}
Then as $\tw(x) \wedge \tw(y)\leq \tw(\ln e)$, we have that $\tw(x) \leq \tw(\ln \tw(y))$ and $\tw(y) \leq \tw(\ln \tw(x))$, so
\begin{align*}
    \tw(z)                      &= \left(\tw(w\wedge \ln w) \vee \tw\left(\ln\left(x\Rightarrow \tw(y)\right)\right) \vee \tw\left(\ln\left(y\Rightarrow \tw(x)\right)\right)\right)\wedge_\tw \left(\tw(y)\to_\tw \tw(x)\right)\\
                                            &\hspace{-0.5cm}= \left(\tw(w\wedge \ln w) \vee \left(\tw(x) \wedge_\tw \tw(\ln\tw(y))\right) \vee \left(\tw(y) \wedge_\tw \tw(\ln\tw(x))\right)\right)\wedge_\tw \left(\tw(y)\to_\tw \tw(x)\right)\\
                                            &\hspace{-0.5cm}= \left(\tw(w\wedge \ln w) \vee \tw(x)  \vee \tw(y) \right)\wedge_\tw \left(\tw(y)\to_\tw \tw(x)\right)\\
                                            &\hspace{-0.5cm}= \left(\tw(x)  \vee \tw(y) \right)\wedge_\tw \left(\tw(y)\to_\tw \tw(x)\right)\\
                                            &\hspace{-0.5cm}= \tw(x)\\
    \tw(\ln z)      &= \left(\tw(w\vee \ln w) \wedge_\tw \left(\tw(x)\to_\tw \tw(y)\right) \wedge_\tw \left(\tw(y)\to_\tw \tw(x)\right)\right)\vee \left(\tw(y) \wedge_\tw \tw(\ln\tw(x))\right)\\
                                            &\hspace{-0.5cm}= \left((\tw(x)\vee \tw(y)) \wedge_\tw \left(\tw(x)\to_\tw \tw(y)\right) \wedge_\tw \left(\tw(y)\to_\tw \tw(x)\right)\right)\vee \tw(y) \\
                                            &\hspace{-0.5cm}= (\tw(x) \wedge_\tw \tw(y))\vee \tw(y)\\
                                            &\hspace{-0.5cm}= \tw(y).
\end{align*}
\end{proof}

Theorem \ref{NT2} provides a way to improve the result from Theorem \ref{adjunction}. First, we show how morphisms behave between Nelson-type algebras.

\begin{lem}Let ${\bf H}_1,{\bf H}_2$ be Brouwerian algebras, $\imath_1\in H_1,\imath_2\in H_2$ and let $F_1\subseteq H_1, F_2\subseteq H_2$ be Boolean filters. If $f:{\bf H}_1\to{\bf H}_2$ is a morphism such that $f(\imath_1)=\imath_2$ and $f(F_1)\subseteq F_2$, then $\varphi_f:{\bf Tw({\bf L}_1,\imath_1,F_1)}\to {\bf Tw({\bf L}_2,\imath_2,F_2)}$ given by $\varphi_f(a,b)=(f(a),f(b))$ is morphism in $\mathcal{NT}$.\end{lem}

\begin{lem}Let ${\bf A}_1,{\bf A}_2\in \mathcal{NT}$. If $\varphi:{\bf A}_1\to{\bf A}_2$ is a morphism, then $f_{\varphi}=\varphi|_{{\bf H}_{{\bf A}_1}}$ is a morphism from ${\bf H}_{{\bf A}_1}$ into ${\bf H}_{{\bf A}_2}$ such that $f_{\varphi}((\ln e_1 \wedge e_1)^2)=(\ln e_2\wedge e_2)^2$ and $f_{\varphi}(F_{A_1})\subseteq F_{A_2}$.\end{lem}

Now, consider the category $\mathcal{BF}_\text{cy}$ with objects triples $(\m H, \imath, F)$, where  $\m H$ is a Brouwerian algebra, $\imath\in H$ and $F\subseteq H$ a Boolean filter; as morphisms we take Brouwerian algebra homomorphisms $f:{\bf H}_1\to{\bf H}_2$ such that $f(\imath_1)=\imath_2$ and $f(F_1)\subseteq F_2$. Note that $\mathcal{NT}$ defines a category where the morphisms are the algebraic homomorphisms. We can conclude:

\begin{thm}\label{NTequivalence} The categories $\mathcal{NT}$ and $\mathcal{BF}_\text{cy}$ are equivalent.\end{thm}

\subsection{Involutive-image $\mathcal{NCA}$    }

In this section we consider a variety in which  the categorical
adjunction in Theorem \ref{adjunction} is strengthened. Unlike in
Nelson-type algebras, the Nelson conucleus $\tw$ doesn't need to be
a term function.

\subsubsection{Twist products and dualizing elements}

Let $\bf L$ be a commutative involutive residuated lattice and let $f=\ln e$. We
define the operation $a\oplus b = \ln ((\ln b)\cdot (\ln a))$. It is well known \cite{GJKO} that this operation is associative, commutative and satisfies, among other things:
$$a\oplus(b\wedge c)=a\oplus b \wedge a\oplus c \  \ \ \mbox{ and  } \ \mbox{
     if } a\leq b \mbox{ then } a\oplus c\leq b\oplus c.$$

\begin{lem}\label{l:filter}
Let $\bf L$ be a commutative involutive residuated lattice, $\imath\in L$ and $F$ be a lattice filter of $L$ that contains the element $e \oplus \imath = f\to \imath$. Then the set
\begin{align*}\text{Tw}(L,\imath,F) = \{(a,b)\in L\times L^\partial: a\cdot b\leq \imath,a\oplus b\in F\}\end{align*}
is the universe of a subalgebra of ${\bf Tw(L,\imath)}$ which is a twist-product over ${\bf L}$.\end{lem}
\begin{proof} Let $(a,b),(a',b')\in \text{Tw}(L,\imath,F)$, then $( a\to f)\to b, (a'\to f)\to b'\in F$.
\begin{itemize}
    \item $\text{Tw}(L,\imath,F)$ is clearly closed under $\ln$ by commutativity of $\oplus$.
    \item $(a,b)\wedge(a',b')=(a\wedge a',b\vee b')$, and
        \begin{align*}(a\wedge a')\oplus(b\vee b') &= a\oplus(b\vee b') \wedge a'\oplus(b\vee b')\geq a\oplus b \wedge a'\oplus b'\in F.\end{align*}
    \item $(a,b)\cdot(a',b')=(a\cdot a',a\to b'\wedge a'\to b)$, and
      \begin{align*}(a\cdot a')\oplus(a\to b'\wedge a'\to b)    &= \left((a\cdot a')\oplus(a\to b')\right) \wedge \left((a\cdot a')\oplus(a'\to b)\right)\\
                                                                                                                            &\hspace{-2cm}=((a\cdot a')\to f)\to(a\to b') \wedge ((a\cdot a')\to f)\to(a'\to b)\\
                                                                                                                            &\hspace{-2cm}=((a\cdot(a\to (a'\to f))\to b'))\wedge ((a'\cdot(a'\to (a\to f))\to b))\\
                                                                                                                            &\hspace{-2cm}\geq ((a'\to f)\to b')\wedge((a\to f)\to b)\in F.\end{align*}
    \item $\tw_{\text{Tw}}(a,b)=(a,a\to\imath)$, and
                \begin{align*}a\oplus(a\to\imath) &= (a\to f)\to (a\to\imath) = (a(a\to f))\to\imath\geq f\to\imath = e \oplus \imath \in F.\end{align*}
\end{itemize}\end{proof}

 Note that $e \oplus \imath \in F$ automatically follows from assuming that $e \in F$ in the special case where $f=0$ is the bottom element; in that case as $e=f\to f$, so $e$ will be the top element.
The next lemma provides a hint on how to define filters starting from the twist structure.

\begin{lem}\label{l:filter2}
Let $\bf L$ be an integral commutative involutive residuated lattice. Consider $\imath\in L$ and  $F$ a lattice filter of $L$. Then $c\in F$ if and only if there exists $(a,b)\in Tw(L,\imath,F)$ such that $c=a \oplus b$.
\end{lem}
\begin{proof} Note that $0 :=\ln e$ is the bottom element.
If $(a,b)\in Tw(L,\imath,F)$, then by definition $c=a\oplus b\in F$. Conversely, if $c\in F$ then $c\oplus 0=c\in F$ and $c\cdot 0=0\leq \imath$, so $(c,0)\in Tw(L,\imath,F)$ satisfies what we wanted. \end{proof}

\subsubsection{Involutive-image Nelson Conucleus Algebras}

If $({\bf A},\tw)\in \mathcal{NCA}$ is such that $\bf A_{\tw}$ is an involutive residuated lattice, we want to find a lattice filter $F_A$ on $\bf A_{\tw}$ such that ${\bf A}\cong {\bf Tw}(\bf A_{\tw},\imath,F_A)$. Given the characterization in Lemma~\ref{l:filter2},
 $F_A$ should contain the elements of the form $a \oplus b$, where $x=(a,b) \in {\bf Tw}(\bf A_{\tw},\imath,F_A)$. Note that $\tw(x)=(a, a \ld \imath)$ and $\tw(\ln x)=(b, b \ld \imath)$. The identification between $a$ and $(a, a \ld \imath)$ and between $b$ and $(b, b \ld \imath)$ provided by $\psi_{(\m L, \imath)}$ guides us in defining
$$F_A = \{\tw(x) \oplus \tw(\ln x): x\in A\}.$$

We define the variety $\mathcal{INCA}$ of algebras $({\bf A},\tw)$ such that $\bf A$ is an involutive commutative residuated lattice with a bottom element $\bot$, $(\hat{\bf A},\tw)\in \mathcal{NCA}$ where $\hat{\bf A}$ is the $\bot$-less reduct of $\bf A$, and such that the following equation holds:
\begin{enumerate}[{(IT}1{)}]
    \item $\tw(\tw(\tw(x)\to\bot)\to\bot)=\tw(x)$.
\end{enumerate}

We note that $\m A$ is also bounded above with top element $\top=\ln \bot= \bot\to\bot$.

In this case, $\bf L_A:= ({\bf A}_{\tw}, 0, 1)$ will be an integral commutative residuated lattice, with $0=\tw(\bot)$ as negation constant element and $1=\tw(\top)=\tw(e)$ as unit element. If we define for all $a \in L_A$ $$\neg_{\tw} a=a\to_{\tw} 0,$$ since $a=\tw(a)$, we have $\neg_\tw a= a \to_\tw 0=\tw(a \Rightarrow 0)= \tw(\tw(a) \to \tw(\bot))=\tw(\tw(a) \to \bot)$. So if $({\bf A},\tw)\in \mathcal{INCA}$, the residuated lattice $\bf L_A:= ({\bf A}_{\tw}, 0, 1)$ is bounded, commutative, and (IT1) translates into: for all $a \in A_\tw$
$$\neg_\tw \neg_\tw a= a.$$

In Lemma~\ref{l:filter}, we saw that if $\bf L$ is a commutative residuated lattice where the dualizing element is the least element and $\imath\in L$, then ${\bf Tw({\bf L},\imath, F)}$ defines a subalgebra of ${{\bf Tw}({\bf L},\imath)}$. Note that the latter is in $\mathcal{INCA}$.

\begin{ex}As a particular case, if $\bf L$ is a bounded commutative residuated lattice satisfying $\neg\neg x=x$, we have that $\tw(x)=x$ is a Nelson conucleus, so $({\bf L},\tw)\in\mathcal{INCA}$ when we consider $\ln x = \neg x$, and by Theorem \ref{Teorema_representacion} we have that ${\bf L}$ embeds into ${\bf Tw}({\bf L},0)$ by the mapping $x\mapsto (x,\neg x)$.\end{ex}

In the notation involving $\bot$, note that we have
\begin{align*}F_A &= \{\tw(x) \oplus \tw(\ln x): x\in A\} = \{(\neg_\tw \tw(x))\to_\tw \tw(\ln x): x\in A\}\\
&=  \{(\tw(x)\to_\tw 0)\to_\tw \tw(\ln x): x\in A\}= \{\tw((\tw(x)\Rightarrow \bot) \Rightarrow \tw(\ln x)): x\in A\}\\
&= \{\tw(\tw(\tw(x)\to \bot) \to \ln x): x\in A\}.\end{align*}
We will need the following technical result.

\begin{lem}\label{filterform} If $({\bf A},\tw)\in \mathcal{NCA}$ has a lower bound $\bot$, then with the previous notation $F_A =\{\tw(\ln z) :\tw(z)=0, z \in A\}$.
\end{lem}
\begin{proof}
Recall that $F_A=  \{\tw(x) \oplus \tw(\ln x): x\in A\}$
Clearly, if $\tw(z)=0$, for some $z \in A$, then $\tw(\ln z)=\tw(z) \oplus \tw(\ln z) \in F_A$. Conversely, let $x \in A$ and define $z=x\cdot \tw(\tw(x)\to \bot)$. We have
$\tw(x) \oplus \tw(\ln x)=\tw(\tw(\tw(x)\to \bot) \to \ln x)
=\tw(\ln(x \cdot \tw(\tw(x)\to \bot))=\tw(\ln z).$
Also,
\begin{align*}\tw(z) &= \tw(x)\cdot \tw(\tw(x)\to \bot) = \tw(\tw(x)\cdot (\tw(x)\to \bot)) = \tw(\bot)=0.\end{align*}
\end{proof}

\begin{thm}If $({\bf A},\tw)\in \mathcal{INCA}$, $1=\tw(e)$ and $\imath=\tw(\ln e)$, then $F_A$ is a lattice filter of $\bf L_A$ and
$${\bf A}\cong {{\bf Tw}({\bf L}_{\bf A},\imath,F_A)}.$$\end{thm}
\begin{proof}
We will use Lemma~\ref{filterform} throughout the proof.
First note that  $1 \oplus \imath=\tw(e) \oplus \tw(\ln e) \in F_A$.
Now, we show that $F_A$ is a lattice filter.
\begin{itemize}
    \item $1\in F_A$, because $\tw(\top) \oplus \tw(\ln \top)=\tw(\top) \oplus \tw(\bot)=1 \oplus 0=1$.
    \item If $a=\tw(\ln x)$, where $\tw(x)=0$ and $b=\tw(\ln y)$, where $\tw(y)=0$, for some $x,y\in A$, define $z:=x \jn y$. Then, $a \mt_\tw b=\tw(\ln x) \mt_\tw \tw(\ln y)=\tw( \ln x \mt \ln y)=\tw(\ln (x \jn y))=\tw(\ln z)$. Also, $\tw(z)=\tw(x \jn y)=\tw(x)\jn \tw(y)=0 \jn 0=0$. So, $a \mt_\tw b\in F_A$.
    \item If $a=\tw(\ln x)$, where $\tw(x)=0$ and $b\geq a$ with $b=\tw(y)$ for some $x,y\in A$, consider
            $z=\ln y \wedge x,$
    so $\ln z=y \jn \ln x$. We have $\tw(\ln z)=\tw(y)\vee \tw(\ln x) = b\vee a=b$, and
            $$\tw(z)=\tw(\ln y \wedge x)=\tw(\ln y) \wedge_\tw \tw(x)=\tw(\ln y) \wedge_\tw 0=0,$$
                        So $b \in F_A$.
\end{itemize}

Now, using Theorem \ref{Teorema_representacion}, we will show that $\phi_{\bf A}(A)=\text{Tw}(L_A,\imath,F_A)$. Note that if $(\tw(x),\tw(\ln x))\in \phi_{\bf A}(A)$, then $\tw(x)\oplus \tw(\ln x)\in F_A$, so $\phi_{\bf A}(A)\subseteq\text{Tw}(L_A,\imath,F_A)$. For the other inclusion, consider $(\tw(x),\tw(y))\in \text{Tw}(L_A,\imath,F_A)$. Therefore we have $\tw(x)\cdot\tw(y)\leq \imath = \tw(\ln e)$ and $\tw(x)\oplus \tw(y)\in F_A$. Thus there exists $w\in A$ such that $\tw(x)\oplus \tw(y)= \tw(\ln w)$ and $\tw(w)=0$. We will find $z\in A$ such that $\tw(z)=\tw(x)$ and $\tw(\ln z)=\tw(y)$.

Recall that $\oplus$ is an operation on $L_A$ defined by $a \oplus b= \neg_\tw a \ra_\tw b$, where for $a \in L_A$ we have $\neg_\tw a = a \ra_\tw 0=\tw(a \Rightarrow 0)= \tw(\tw(a) \ra \tw(\bot))= \tw(\tw(a) \ra \bot)$.
For convenience we extend these operations to $A$ by defining for $x,y \in A$,
$\neg_\tw x=\tw(\tw(x) \ra \bot)$ and $x \oplus y=  \neg_\tw x \ra_\tw \tw(y)$. It then follows that $\tw(x \oplus y)=\tw( \neg_\tw x \ra_\tw \tw(y))
=\tw( \neg_\tw x \Rightarrow y)
=\tw( \neg_\tw x) \Rightarrow \tw(y)
= \neg_\tw \tw(x) \Rightarrow \tw(y)
=\tw(x) \oplus \tw(y)$.

Consider $z = x \oplus w \mt \ln \tw(y)$. Then recalling that $\tw(x)\cdot\tw(y)\leq \tw(\ln e)$ implies that $\tw(x)\leq \tw(\ln\tw(y))$
    \begin{align*}
        \tw(z) &= \tw(x \oplus w \mt \ln \tw(y)) = \tw(x) \oplus \tw(w) \mt_\tw \tw(\ln \tw(y)))\\
      &= \tw(x) \oplus 0 \mt_\tw \tw(\ln \tw(y))) = \tw(x)\wedge_\tw \tw(\ln\tw(y)) = \tw(x),
        \end{align*}
            and
         \begin{align*}\tw(\ln z)
            &=  \tw( \ln(\neg_\tw \ra w) \vee\tw(y)) =  \tw(\ln(\neg_\tw x \ra w)) \vee \tw(y)\\
    &=  \tw(\neg_\tw x\cdot \ln w)) \vee \tw(y) =  \tw(\neg_\tw x) \cdot \tw(\ln w)) \vee \tw(y)\\
    &=  \tw(\neg_\tw x) \cdot (\tw(x) \oplus \tw(y))) \vee \tw(y) =  \neg_\tw \tw(x) \cdot (\tw(x) \oplus \tw(y))) \vee \tw(y) = \tw(y).
        \end{align*}
          
\end{proof}

As before the isomorphism on objects extends to a categorical equivalence.

\medskip
\footnotesize 

Manuela Busaniche

Departamento de Matem\'atica, Facultad de Ingenier\'ia Qu\'imica, Universidad Nacional del Litoral - CONICET

Santiago del Estero 2829

Santa Fe, Argentina

\textit{E-mail address:} \texttt{mbusaniche@santafe-conicet.gov.ar}\\

Nikolaos Galatos

University of Denver

2390 S. York St.

Denver, USA

\textit{E-mail address:} \texttt{ngalatos@du.edu}\\

Miguel Andr\'es Marcos

Departamento de Matem\'atica, Facultad de Ingenier\'ia Qu\'imica, Universidad Nacional del Litoral - CONICET

Santiago del Estero 2829

Santa Fe, Argentina

\textit{E-mail address:} \texttt{mmarcos@santafe-conicet.gov.ar}

\end{document}